\documentclass[a4paper,11pt]{amsart}
\usepackage[english]{babel}
\usepackage{amsfonts,amssymb,mathrsfs}
\usepackage[bookmarks=true]{hyperref}
\usepackage[all]{xy}
\newtheorem{theorem}{Theorem}[section]
\newtheorem{corollary}[theorem]{Corollary}
\newtheorem{lemma}[theorem]{Lemma}
\newtheorem{proposition}[theorem]{Proposition}
\theoremstyle{definition}
\newtheorem{definition}[theorem]{Definition}

\newtheorem{example}[theorem]{Example}
\newtheorem{remark}[theorem]{Remark}

\numberwithin{equation}{section}

\newcommand{\AAA}{\mathbf A}
\newcommand{\GG}{\mathbf G}
\newcommand{\NN}{\mathbf N}
\newcommand{\PP}{\mathbf P}
\newcommand{\QQ}{\mathbf Q}
\newcommand{\RR}{\mathbf R}
\newcommand{\TT}{\mathbf T}
\newcommand{\VV}{\mathbf V}
\newcommand{\ZZ}{\mathbf Z}
\newcommand{\SL}{\mathbf{SL}}
\newcommand{\SP}{\mathbf{Sp}}

\newcommand{\SA}{\mathscr{A}}
\newcommand{\SC}{\mathscr{C}}
\newcommand{\SD}{\mathscr{D}}
\newcommand{\SI}{\mathscr{I}}
\newcommand{\SN}{\mathscr{N}}
\newcommand{\ST}{\mathscr{T}}
\newcommand{\SU}{\mathscr{U}}
\newcommand{\SV}{\mathscr{V}}
\newcommand{\SX}{\mathscr{X}}
\newcommand{\SY}{\mathscr{Y}}
\newcommand{\SZ}{\mathscr{Z}}
\newcommand{\lra}{\longrightarrow}
\newcommand{\hb}{{\rm H}_{\rm b}}
\newcommand{\hc}{{\rm H}_{\rm c}}
\newcommand{\hh}{{\rm H}}

\newcommand{\Is}{\mathrm{Is}}
\newcommand{\id}{\mathrm{Id}}
\newcommand{\ee}{{\mathrm{E}}}
\newcommand{\se}{\subseteq}
\newcommand{\sm}{\setminus}
\newcommand{\lft}{L^\infty}
\newcommand{\lw}{L^\infty_{\rm w*}}
\newcommand{\bu}{\bullet}

\newcommand{\cont}{\mathrm{C}}

\newcommand{\cw}{\cont_{\rm w*}}
\newcommand{\cm}{\mathcal{C}}
\newcommand{\rank}{\mathrm{rank}}
\newcommand{\minrank}{\mathrm{mink}}
\newcommand{\Rad}{\mathrm{Rad}}
\newcommand{\ti}{\widetilde}

\newcommand{\vn}{\varnothing}
\begin{document}
\title[Semi-simple groups and bounded cohomology]{On the bounded cohomology of semi-simple groups,\\ S-arithmetic groups and products}
\author{Nicolas Monod}
\address{Universit\'e de Gen\`eve}
\email{nicolas.monod@math.unige.ch}
\thanks{Supported in part by Fonds National Suisse}
\begin{abstract}
We prove vanishing results for Lie groups and algebraic groups (over any local field) in bounded cohomology. The main result is a vanishing below twice the rank for semi-simple groups. Related rigidity results are established for S-arithmetic groups and groups over global fields. We also establish vanishing and cohomological rigidity results for products of general locally compact groups and their lattices.
\end{abstract}
\maketitle
\let\languagename\relax  
\section{Introduction and statement of the results}
\subsection{Motivation}
The main object of this paper is to study the bounded cohomology of arithmetic groups and semi-simple Lie or algebraic groups. On the one hand, the results presented below can be approached by comparison with the classical vanishing theorems for semi-simple groups due to Borel--Wallach~\cite{Borel-Wallach}, G.~Zuckerman~\cite{Zuckerman} and W.~Casselman~\cite{Casselman} and with the classical question of invariance of the cohomology of arithmetic groups (\emph{e.g.}\ A.~Borel~\cite{Borel74} and Borel--Serre~\cite{Borel-Serre73}). On the other hand, this study is motivated by the growing array of applications of previous vanishing or non-vanishing results in bounded cohomology ever since M.~Gromov's seminal work~\cite{Gromov} (see~\cite{MonodICM} for a recent survey).

\smallskip

There is a degree of similarity with classical statements: in ordinary cohomology, one has vanishing below the rank for suitable non-trivial representations (V.3.3 and XI.3.9 in~\cite{Borel-Wallach}), whilst we establish vanishing below twice the rank. None of the classical methods, however, apply: the cohomology of semi-simple groups can for instance be reduced to Lie-algebraic questions by the van~Est isomorphism in the Lie case, and to twisted simplicial cohomology on the discrete Bruhat--Tits building in the non-Archimedean case (compare also Casselman--Wigner~\cite{Casselman-Wigner}). There are \emph{a priori} obstructions for any such method to work in bounded cohomology; even one-dimensional objects such as the free group are known to have bounded cohomology at least up to degree three. Indeed, no non-trivial dimension bound in bounded cohomology for any group is known. Still, we do make use of spherical Tits buildings.

\bigskip

Arithmetic groups occur as lattices in semi-simple groups and thus there is a corresponding \emph{restriction map} in cohomology; its image (for real coefficients) is called the \emph{invariant} part of the cohomology of the arithmetic group. This is the least mysterious part, since it comes ultimately from (primary and secondary) characteristic classes in the Lie case and vanishes altogether otherwise. The extent to which the cohomology of arithmetic groups is invariant has therefore been the focus of thorough investigations; the results depend highly on the type of the lattice and, in positive $\QQ$-rank, invariance holds only in lower degrees, \emph{e.g.}\ of the order $n/4$ for $\SL_n(\ZZ)$ (A.~Borel, Thm.~7.5 in~\cite{Borel74}). (For surveys, see J.-L.~Koszul~\cite{Koszul68}, J.-P.~Serre~\cite{Serre71}, J.~Schwermer~\cite{Schwermer90}.)

\smallskip

Eckmann--Shapiro induction provides a link between these questions and the cohomology of semi-simple groups with \emph{induction modules}. The difficulties in positive $\QQ$-rank are reflected in that the corresponding induction modules are not unitarisable in this case. In bounded cohomology, it turns out that invariance holds below twice the rank regardless of the type of the lattice; for instance, up to $2n-3$ for $\SL_n(\ZZ)$. This comes as at a purely functional-analytic cost and leads us to work with a class of non-separable non-continuous modules that we call \emph{semi-separable}. As pointed out to us by M.~Burger, our invariance results show in particular that the comparison map to ordinary cohomology fails to be an isomorphism in many higher rank cases.

\bigskip

Results in degree two were previously obtained in~\cite{Burger-Monod1},\cite{Burger-Monod3},\cite{Monod}; see also~\S6 in~\cite{Monod-Shalom1} which used already non-minimal parabolics to establish vanishing in degree two. For the special case of $\SL_n$, we presented vanishing (up to the rank only) in the note~\cite{MonodMRL} using special (mirabolic) maximal parabolics following our earlier work~\cite{MonodJAMS}. It has been established by T.~Hartnick and P.-E.~Caprace independently that these methods do not apply to other classical groups (personal communications).

\subsection{Semi-simple groups}
We begin with a statement for \emph{simple} groups in order to prune out excessive terminology. A \emph{local field} is any non-discrete locally compact field (thus including the Archimedean case and arbitrary characteristics).

\begin{theorem}\label{thm:simplesimple}
Let $k$ be a local field, $\,\GG$ a connected simply connected almost $k$-simple $k$-group and $V$ a continuous unitary $\,\GG(k)$-representation not containing the trivial one. Then
$$\hb^n(\GG(k); V)=0 \kern1cm\text{for all \ $n < 2\,\rank_k(\GG)$.}$$
This vanishing holds more generally for any semi-separable coefficient $\,\GG(k)$-module $V$ without $\,\GG(k)$-invariant vectors.
\end{theorem}

The definition of semi-separability (Section~\ref{sec:semi-separable}) does not impose any additional restriction on the underlying Banach space of coefficient modules, which are by definition the dual of a continuous separable module. The generality of semi-separability allows our results to apply notably to certain (non-separable, non-continuous) modules of $\lft$-functions that are useful for studying lattices and for applications to ergodic theory. Further, all representations in our statements can be uniformly bounded instead of isometric.

In the absence of semi-separability, the statement can indeed fail (Example~\ref{ex:contre-ex}). We also point out that vanishing fails for the trivial representation; $\hb^2(-,\RR)$ has been determined in~\cite{Burger-Monod1} and extensively used \emph{e.g.}\ in~\cite{Burger-Iozzi-Wienhard03},\cite{Burger-Iozzi-Labourie-Wienhard}.

\bigskip

We now give the general statement for semi-simple groups over possibly varying fields; this generality is necessary \emph{e.g.}\ to view S-arithmetic groups as lattices. Let thus $\{\GG_\alpha\}_{\alpha\in A}$ be connected simply connected semi-simple $k_\alpha$-groups, where $\{k_\alpha\}_{\alpha\in A}$ are local fields and $A$ is a finite non-empty set. Consider the locally compact group $G=\prod_{\alpha\in A}\GG_\alpha(k_\alpha)$ defined using the Hausdorff $k_\alpha$-topologies. As usual, write
$$\rank(G)=\sum_{\alpha\in A} \rank_{k_\alpha}(\GG_\alpha).$$
Define further $\minrank(G)$ as the minimal rank of almost $k_\alpha$-simple factors of all $\GG_\alpha$.

\begin{theorem}\label{thm:semi-simple}
Let $V$ be a semi-separable coefficient $G$-module. (i)~If there are no fixed vectors for any of the almost $k_\alpha$-simple factors of any $\GG_\alpha$, then
\begin{align*}
\hb^n(G; V)&=0 \kern1cm\text{for all \ $n < 2\,\rank(G)$.}\\
\intertext{(ii)~If we only assume $V^G=0$, then still}
\hb^n(G; V)&=0 \kern1cm\text{for all \ $n < 2\,\minrank(G)$.}
\end{align*}
\end{theorem}

\noindent
Again, the assumptions on $V$ are necessary.

\begin{remark}[{{\bfseries General algebraic groups}}]\label{rem:general_groups}
The above theorem implies immediately a vanishing result for general algebraic groups, but the statement becomes somewhat more convoluted: Let $G$ be the group of rational points, $R<G$ its amenable radical and $V$ a semi-separable coefficient $G$-module. One has isomorphisms $\hb^\bu(G; V)\cong \hb^\bu(G; V^R) \cong \hb^\bu(G/R; V^R)$, see \emph{e.g.}~\cite[8.5.2, 8.5.3]{Monod}. Now $G/R$ is semi-simple, and we may pass to the (algebraic) connected component of the identity without affecting vanishing. To account for the lack of simple connectedness, one replaces the assumption on invariant vectors of the almost simple factors (of $G/R$ in $V^R$) by the corresponding assumption on their canonical normal cocompact subgroups as defined in~\cite[\S6]{Borel-Tits73} (see also~I.1.5 and~I.2.3 in~\cite{Margulis}). As an illustration of this procedure in the related case of Lie groups, see Corollary~\ref{cor:restriction_R}.
\end{remark}

\subsection{Lattices in semi-simple groups}
Turning to lattices, the next result shows that the real bounded cohomology of S-arithmetic groups is invariant below twice the rank.

\begin{corollary}\label{cor:restriction}
Let $G$ be as for Theorem~\ref{thm:semi-simple} and let $\Gamma<G$ be an irreducible lattice. Then the restriction map
$$\hb^n(G; \RR) \lra \hb^n(\Gamma; \RR)$$
is an isomorphism for all $\ n < 2\,\rank(G)$.
\end{corollary}

In particular, the real bounded cohomology of S-arithmetic groups is completely determined by the corresponding algebraic group below twice the rank. It is perhaps surprising that in spite of the vanishing results for non-trivial representations, we still know rather little about $\hb^n(G; \RR)$, even for $G=\SL_n(\RR)$ (some results are in~\cite{Burger-MonodERN},\cite{MonodJAMS}).

\medskip

In the Archimedean case, we can use \'E.~Cartan's correspondence between semi-simple Lie groups and symmetric spaces to deduce a result on the singular bounded cohomology of locally symmetric spaces. In order to state it, we need to introduce a generalisation of the notion of orientability and orientation cover. We call a connected locally symmetric space \emph{ample} if its fundamental group lies in the neutral component $\Is(\ti X)^\circ$ of the isometry group of its universal cover. Every locally symmetric space has a canonical smallest finite cover that is ample.

\begin{corollary}\label{cor:spaces}
Let $X$ be a connected ample finite volume irreducible Riemannian locally symmetric space of non-compact type. Then the real singular bounded cohomology $\hb^n(X)$ is canonically isomorphic to $\hb^n(\Is(\ti X)^\circ)$ for all $\ n < 2\,\rank(X)$.
\end{corollary}

\noindent
(In particular, it follows that real singular bounded cohomology is stable along the inverse system of finite covers of $X$.)

\bigskip

With non-trivial coefficients, we have results for lattices that are in direct analogy with the above Theorems~\ref{thm:simplesimple} and~\ref{thm:semi-simple}, and indeed can be deduced from them thanks to the flexibility afforded by semi-separability. We first consider lattices in simple groups:

\begin{corollary}\label{cor:lattice_simple}
Let $k$ be a local field, $\GG$ a connected simply connected almost $k$-simple $k$-group, $\Gamma< \GG(k)$ a lattice and $W$ any semi-separable coefficient $\Gamma$-module $W$ without $\Gamma$-invariant vectors. Then
$$\hb^n(\Gamma; W)=0 \kern1cm\text{for all \ $n <  2\,\rank_k(\GG)$.}$$
\end{corollary}

\noindent
Combining this with (a minor variation on) Corollary~\ref{cor:restriction}, one finds that without any assumption on the semi-separable coefficient $\Gamma$-module $W$ we have isomorphisms
$$\hb^n(\Gamma; W)\ \cong\ \hb^n(\Gamma; W^\Gamma)\ \cong\ \hb^n(G; W^\Gamma)$$
induced by the inclusion $W^\Gamma\to W$ and the restriction, respectively.

\medskip

For general semi-simple groups and arbitrary semi-separable coefficient $\Gamma$-module, the statement takes the form of a \emph{rigidity} result:

\begin{corollary}\label{cor:lattice_semisimple}
Let $G$ be as for Theorem~\ref{thm:semi-simple}, let $\Gamma<G$ be an irreducible lattice and let $W$ be a semi-separable coefficient $\Gamma$-module. If $\hb^n(\Gamma; W)$ is non-zero for some $n <  2\,\rank_k(\GG)$, then the $\Gamma$-representation extends continuously to $G$ upon possibly passing to a non-trivial closed invariant subspace of $W$.
\end{corollary}

In parallel to the second part of Theorem~\ref{thm:semi-simple}, there is a stronger result but in lower degrees:

\begin{corollary}\label{cor:lattice_semisimple_mink}
Let $G$ be as for Theorem~\ref{thm:semi-simple}, $\Gamma<G$ any lattice and $W$ a semi-separable coefficient $\Gamma$-module without invariant vectors. Then
$$\hb^n(\Gamma; W)=0 \kern1cm\text{for all \ $n <  2\,\minrank(\GG)$.}$$
\end{corollary}

\subsection{Product groups and their lattices}
In various settings, rigidity results known for semi-simple groups of higher rank have later been established for products of general locally compact groups, underscoring a strong analogy (see \emph{e.g.}\ \cite{Burger-Mozes2},\cite{Shalom00},\cite{MonodCAT0}). Accordingly, one finds in~\cite{Burger-Monod3} rigidity results for~$\hb^2$ of general products. Even for \emph{products of discrete groups}, those results have found surprising applications~\cite{Monod-Shalom2}.

We shall establish statements for products that are analogous to Theorem~\ref{thm:semi-simple} and its consequences; in the situation at hand, the product case is simpler to prove indeed.

\begin{theorem}\label{thm:splitting}
Let $\ G= G_1\times \cdots\times G_\ell$ be a product of locally compact second countable groups and $\ V$ be a semi-separable coefficient $G$-module. If $\ V^{G_i}=0$ for all $\ 1\leq i\leq \ell$, then
$$\hb^n(G; V)=0 \kern10mm \forall\ n < 2\ell.$$
\end{theorem}

\noindent
Again, semi-separability cannot be disposed of, see Example~\ref{ex:contre-ex}. The assumption $V^{G_i}=0$ is also needed.

\medskip

A lattice $\Gamma$ in $G= G_1\times \cdots\times G_\ell$ shall be called \emph{irreducible} if $G_j\cdot\Gamma$ is dense in $G$ for all $j$. By a result of G.~Margulis~\cite[II.6.7]{Margulis}, this definition coincides with the classical notion of irreducibility for lattices in semi-simple groups as long as no $G_i$ is compact; for details see Remark~\ref{rem:irreducibility}. Theorem~\ref{thm:splitting} can be used to prove a superrigidity result for representations of irreducible lattices $\Gamma<G$. To this end, let $W$ be any coefficient $\Gamma$-module. Define $W_{G_j}$ to be the collection of all $w\in W$ for which the corresponding orbit map $\Gamma\to W$ extends continuously to $G$ and factors through $G/G_j$. The irreducibility condition shows that $W_{G_j}$ has a natural Banach $G$-module structure (Section~\ref{sec:lattices}). Of course, one expects $W_{G_j}=0$ in general; $W_{G_j}\neq 0$ is precisely a superrigidity statement.

\begin{corollary}\label{cor:split_lattice}
Let $\ G= G_1\times \cdots\times G_\ell$ be a product of $\ell\geq 2$ locally compact second countable groups, $\ \Gamma<G$ an irreducible lattice and $W$ be any semi-separable coefficient $\Gamma$-module.

If $\ \hb^n(\Gamma; W)\neq 0$ for any $\ n < 2\ell$, then $W_{G_j}\neq 0$ for some $G_j$.
\end{corollary}

Just as Theorem~\ref{thm:semi-simple} implies the restriction isomorphism of Corollary~\ref{cor:restriction}, we can use Theorem~\ref{thm:splitting} to establish that the real bounded cohomology of irreducible lattices is invariant in a suitable range:

\begin{corollary}\label{cor:rstriction_product}
Let $\ G= G_1\times \cdots\times G_\ell$ be a product of locally compact second countable groups and $\Gamma<G$ an irreducible lattice. Then the restriction map
$$\hb^n(G; \RR) \lra \hb^n(\Gamma; \RR)$$
is an isomorphism for all $\ n < 2\ell$.
\end{corollary}

\subsection{Global fields and ad\'elic groups}\label{sec:intro_adelic}
Recall that a \emph{global field} is a finite extension either of the rationals $\QQ$ (a number field), or a field of rational functions in one variable over a finite field. The ring $\AAA_K$ of ad\`eles is the restricted product of the completions of $K$. In particular, there is a diagonal embedding $K\to \AAA_K$ realizing $K$ as \emph{principal ad\`eles}.

\smallskip

The strong approximation property places ad\'elic groups within the scope of our results for products; the following rigidity result does not have any restriction on the degree $n$.

\begin{theorem}\label{thm:adelic_rigidity}
Let $K$ be a global field and $\GG$ a connected simply connected almost $K$-simple $K$-group. Let $W$ be any semi-separable coefficient $\GG(K)$-module.

If $\ \hb^n(\GG(K);W)\neq 0$ for some $n$, then the $\GG(K)$-representation extends continuously to $\GG(\AAA_K)$ upon possibly passing to some non-trivial closed sub-module of $W$.
\end{theorem}

\noindent
In fact, we will see that the extended representation in the above result can be assumed to be trivial on infinitely many local factors.

\smallskip

The restriction map corresponding to the embedding $K\to \AAA_K$ yields an isomorphism as in Corollary~\ref{cor:restriction}, but in all degrees (for degree two, see Theorem~28 in~\cite{Burger-Monod3}).

\begin{corollary}\label{cor:adelic_restriction}
The restriction map
$$\hb^\bu(\GG(\AAA_K);\RR) \lra \hb^\bu(\GG(K);\RR)$$
is an isomorphism in all degrees.
\end{corollary}

When $K$ has characteristic zero, Borel--Yang show in~\cite{Borel-Yang} that the usual cohomology $\hh^\bu(\GG(K);\RR)$ is isomorphic to the cohomology of the product of Archimedean completions of $\GG(K)$. In order to deduce from Corollary~\ref{cor:adelic_restriction} a corresponding result in bounded cohomology, we would need a vanishing for $p$-adic groups and trivial coefficients. The strategy of Borel--Yang is very different from ours, as they use a limiting argument whilst approximating $\GG(K)$ with S-arithmetic groups; this allows them to apply the main result of Blasius--Franke--Grunewald~\cite{Blasius-Franke-Grunewald} (or A.~Borel~\cite{Borel76} in the $K$-anisotropic case; the positive characteristic analogue is due to G.~Harder~\cite{Harder77}).

\subsection{Comments on the proofs}
The proof of Theorems~\ref{thm:simplesimple} and~\ref{thm:semi-simple} can be simplified provided (i)~one considers continuous unitary representations (or \emph{separable} coefficient modules, which are automatically continuous~\cite[3.3.2]{Monod}) and (ii)~one remains below the rank rather than twice the rank. We suggest to the reader to keep this setting in mind as a guide to reading the general semi-separable case; the additional difficulties of the general situation are justified first and foremost by the fact that they are the key to the results for lattices.

\smallskip

More specifically, here is a very brief outline of the proof in the simpler case. Let $G=\GG(k)$ be a simple group as in Theorem~\ref{thm:simplesimple} and $V$ be a continuous unitary $G$-representation without invariant vectors. An investigation (following Borel--Serre~\cite{Borel-Serre76}) of the topologised Tits building of $G$ establishes a topological analogue of the Solomon--Tits theorem~\cite{Solomon}. (Recall that the latter states that, as abstract simplicial complexes, such buildings have the homotopy type of a bouquet of spheres; this purely combinatorial statement is also exposed by H.~Garland in~\cite[App.~2]{Garland}.) The topological analogue can be formulated as an exact sequence of sums of continuous $V$-valued function spaces of the form $\cont(G/Q;V)$, where $Q$ ranges over all standard parabolic subgroups of $G$. If $r$ is the rank of $G$, there are $2^r$ such parabolics and the function spaces fit into an exact sequence of length $r$. Since $V$ is a \emph{continuous} $G$-module, there is an isomorphism between $\cont(G/Q;V)$ and (a continuous avatar of) the induced module associated to $V$ viewed as a $Q$-module. This leads to investigating the cohomology $\hb^\bu(Q;V)$. The boundedness of bounded cochains allows to factor out the amenable radical of $Q$, and the classical Mautner phenomenon then implies that the latter cohomology vanishes. This provides vanishing below the rank $r$ in this setting.

\smallskip

In order to deal with the general case, one possibility is to seek a measurable version of the Solomon--Tits theorem, as $\lft$-induction holds for all coefficient modules. At first sight, it is not clear whether the arguments can be adapted; indeed, the idea behind Solomon--Tits is to retract to a point all the building except for the chambers opposite a Weyl chamber, but this locus is a null-set. However, as a consequence of a discussion of coefficient modules (Section~\ref{sec:complements_modules}), it turns out that the corresponding cohomological result still holds (Theorem~\ref{thm:complexe:lw}). The flexibility of measurable induction also allows us to double the rank by running the arguments for the spherical join of two (opposite) copies of the building.

\medskip

For our results about products, we shall propose a rudimentary analogue of the Tits building for arbitrary products, namely a ``spherical join'' of Poisson boundaries (Section~\ref{sec:products}). The intuition here is that if $B$ is the Poisson boundary of a random walk on any group $G$ and $B^-$ the boundary of the associated backward walk, then there is a one-dimensional simplicial complex $B*B^-$ which has some aspects of a (doubled) Tits building. In presence of a product of $\ell$ factors, the join of these spaces is an object of dimension $2\ell-1$.

\bigskip

We shall try to introduce all needed notation and background. For more details on the (relative) theory of semi-simple groups, we refer to Borel--Tits~\cite{Borel-Tits65},\cite{Borel-Tits72} or G.~Margulis~\cite[Chap.~I]{Margulis}. For more background on bounded cohomology, see~\cite{Burger-Monod3} and~\cite{Monod}.

\section{Continuous simplicial cohomology and Tits buildings}
In this section, we consider an elementary and rather unrestricted notion of topologised simplicial complexes for which one can define continuous cohomology; we then compute this cohomology in the case of Tits buildings. One could work instead within the (well-known) framework of simplicial objects in the category of spaces. Since all the theory we need is defined and proved below in two pages, we feel that the additional structure and restrictions of simplicial objects would be a burden; this accounts for our simple-minded approach.

\smallskip

The result for Tits buildings presented in Section~\ref{sec:tits} is based on the work of Borel--Serre~\cite[\S1--3]{Borel-Serre76}; we need a more general setting than treated in~\cite{Borel-Serre76} in order to deal with non-discrete modules and mixed fields; thus we present a complete proof. Our topological approach follows however closely the algebraic arguments of Borel--Serre; we caution the reader that the sign conventions are different (the conventions of~\cite{Borel-Serre76} being non-simplicial).

\subsection{Continuous simplicial cohomology}
Recall that an \emph{abstract simplicial complex} is a collection $\SX$ of finite non-empty sets which is closed under taking non-empty subsets. We write $\SX^{(n)}$ for the collection of \emph{$n$-simplices}, namely sets of cardinality $n+1$, and abuse notation by identifying $\SX^{(0)}$ with the union of its elements. (We shall also occasionally call \emph{simplex} the simplicial complex of all non-empty subsets of a given finite set.)

To define cohomology, it is convenient to introduce orientations; but the following weaker structure serves the same purpose.

\begin{definition}
A \emph{sufficient orientation} on an abstract simplicial complex $\SX$ is a partial order on $\SX^{(0)}$ which induces a total order on every simplex.
\end{definition}

\noindent
For instance, a (numbered) partition $\SX^{(0)}= \SX^{(0,1)}\sqcup \ldots \sqcup \SX^{(0,r)}$ defines a sufficient orientation provided every simplex contains at most one element of each $\SX^{(0,j)}$; this forces $\dim(\SX)\leq r-1$.

\medskip

One defines the \emph{face maps}
$$\partial_{n,j}: \SX^{(n)} \lra \SX^{(n-1)} \kern10mm(n\geq 1, 0\leq j \leq n)$$
by removing the element $x_j$ from a simplex $\{x_0<\ldots < x_n\}$; it follows $\partial_{n-1,i}\partial_{n,j}=\partial_{n-1,j-1}\partial_{n,i}$ for all $0\leq i<j\leq n$.

\begin{definition}
A sufficiently oriented topologised simplicial complex, or \emph{sot complex}, is an abstract simplicial complex $\SX$ endowed with a topology and a sufficient orientation such that all face maps are continuous.
\end{definition}

We emphasize the difference with \emph{realizations} of $\SX$; in the present case, we simply have a topology on the set $\SX=\bigsqcup_{n\geq 0} \SX^{(n)}$ of simplices. Our simple-minded definition of sot complexes does not even impose that $\SX^{(n)}$ be closed in $\SX$.

\medskip

A \emph{morphism} of sot complexes is a simplicial map that is continuous on each $\SX^{(n)}$ and is compatible with the sufficient orientations. The continuous simplicial cohomology is defined exactly as in the abstract simplicial case but adding the continuity requirement:

\begin{definition}\label{defi:cont_coho}
Let $\SX$ be a sot complex and $V$ an Abelian topological group. Consider the space $\cont(\SX^{(n)};V)$ of all continuous maps (\emph{cochains}) with the convention $\cont(\vn;V)=0$. Define the coboundary operators
$$d_n: \cont(\SX^{(n-1)};V)\lra \cont(\SX^{(n)};V),\kern10mm d_n=\sum_{j=0}^n (-1)^j\partial_{n,j}^*.$$
The \emph{continuous simplicial cohomology} $\hc^n(G;V)$ is defined as $\ker(d_{n+1})/\mathrm{range}(d_n)$.
\end{definition}

An example of a morphism of sot complexes is provided by the inclusion map $\SX_0\to\SX$ of a subcomplex $\SX_0$ of $\SX$. In this situation, one defines the \emph{relative} continuous simplicial cohomology $\hc^\bu(\SX,\SX_0;V)$ as the cohomology of the subcomplex of $\cont(\SX^{(\bu)};V)$ consisting of the spaces
$$\cont(\SX^{(n)},\SX^{(n)}_0 ;V)\ =\ \big\{f\in \cont(\SX^{(n)};V)\ :\ f|_{\SX^{(n)}_0}=0 \big\}$$
A subcomplex $\SX_0\se\SX$ is said \emph{closed} if each $\SX_0^{(n)}$ is closed in $\SX^{(n)}$.

\begin{lemma}\label{lemma:relative}
Let $\SX$ be a metrisable sot complex and $\SX_0\se \SX$ a closed subcomplex. Then there is a natural long exact sequence
$$\cdots \lra \hc^\bu(\SX, \SX_0;V) \lra \hc^\bu(\SX;V) \lra \hc^\bu(\SX_0;V) \lra \hc^{\bu+1}(\SX, \SX_0;V) \lra \cdots$$
for any locally convex topological vector space $V$.
\end{lemma}

\begin{proof}
In view of the classical ``snake lemma'', we need to show that
$$0\lra \cont(\SX^{(n)},\SX_0^{(n)};V) \xrightarrow{\ \iota\ } \cont(\SX^{(n)};V) \xrightarrow{\ r\ } \cont(\SX_0^{(n)};V) \lra 0$$
is an exact sequence for all $n$, where $\iota$ is the inclusion and $r$ the restriction. The point at issue is surjectivity on the right. This amounts to the fact that $V$ is a universal extensor for metric spaces, i.e. to a $V$-valued generalisation of Tietze's extension theorem. The latter is the content of J.~Dugundji's Theorem~4.1 in~\cite{Dugunji51}.
\end{proof}

\noindent
(Dugundji's construction provides in fact a right inverse for $r$ which is linear and of norm one, compare~\cite[7.1]{Michael53}.)

\medskip

We shall need the following form of excision.

\begin{lemma}\label{lemma:ex}
Let $f:\SX\to \SY$ be a surjective morphism of sot complexes and $\SX_0\se \SX$, $\SY_0\se \SY$ subcomplexes with $f(\SX_0)=\SY_0$. Assume that $f$ is injective on $\SY\sm \SY_0$ and is open. Then $f$ induces an isomorphism
$$\hc^\bu(\SY,\SY_0;V)\xrightarrow{\ \cong\ } \hc^\bu(\SX,\SX_0;V)$$
for any Abelian topological group $V$.
\end{lemma}

\begin{proof}
The map $f$ induces an isomorphism already at the cochain level.
\end{proof}

Here is a trivial example of a sot complex:

\begin{example}\label{exo:product}
Let $\SC$ be an abstract simplicial complex and $Z$ a topological space. Consider for every $n$ the space $\SX^{(n)} = \SC^{(n)}\times Z$ endowed with the product topology, wherein $\SC^{(n)}$ has the discrete topology. Any orientation of $\SC$ provides a sufficient orientation on $\SX$, with inclusions and face maps defined purely on the first variable. In other words, the abstract simplicial structure on $\SX$ amounts to considering the product with $Z$ considered as a $0$-complex. Abusing notation, we write $\SX=\SC\times Z$. The identification of $\cont(\SX^{(n)};V)$ with $\cont(\SC^{(n)};\cont(Z;V))$ induces a natural isomorphism
$$\hc^\bu(\SX;V)\xrightarrow{\ \cong\ } \hh^\bu(\SC;\cont(Z;V))$$
for any Abelian topological group $V$. We shall only use this construction in case $\SC$ is finite and contractible.
\end{example}

\subsection{The case of Tits buildings}\label{sec:tits}
Let $\{k_\alpha\}_{\alpha\in A}$ be a family of local fields, where $A$ is a finite non-empty set. For each $\alpha$, let $\GG_\alpha$ be a connected simply connected semi-simple $k_\alpha$-group, $\TT_\alpha$ a maximal $k_\alpha$-split torus, $W_\alpha=\SN_{\GG_\alpha}(\TT_\alpha)/\SZ_{\GG_\alpha}(\TT_\alpha)$ the associated relative Weyl group. (Where $\SN$ and $\SZ$ denote normalisers and centralisers respectively.) Let $\PP_\alpha$ be a minimal parabolic $k_\alpha$-group in $\GG_\alpha$ containing $\TT_\alpha$. Let $S_\alpha$ be the simple roots of $\GG_\alpha$ relatively to $\TT_\alpha$ associated to $\PP_\alpha$; we identify $S_\alpha$ with the corresponding set of reflections in $W_\alpha$. Recall that to every subset $I\se S_\alpha$ one associates the parabolic $k_\alpha$-group $\PP_{\alpha,I}$ generated by $\PP_\alpha$ and the centraliser $\SZ_{\GG_\alpha}(\TT_{\alpha,I})$ of $\TT_{\alpha,I}=(\bigcap_{a\in I}\ker a)^\circ$; in particular, $\PP_{\alpha,\vn}=\PP_\alpha$ and $\PP_{\alpha,S_\alpha}=\GG_\alpha$. Recall also that $S_\alpha$ contains $\rank_{k_\alpha}(\GG_\alpha)$ elements. We assume throughout this section that $\GG_\alpha$ has no $k_\alpha$-anisotropic factors, which implies in particular $S_\alpha\neq\vn$. All spaces of $k_\alpha$-points will be endowed with the $k_\alpha$-topology.

\medskip

We now consider the locally compact group $G=\prod_{\alpha\in A}\GG_\alpha(k_\alpha)$ and define $W=\prod_{\alpha\in A}W_\alpha$, $S=\bigsqcup_{\alpha\in A} S_\alpha$, $r=\sum_{\alpha\in A}\rank_{k_\alpha}(\GG_\alpha)=|S|$. By abuse of language, we call \emph{parabolic} all subgroups $Q<G$ of the form $Q=\prod_{\alpha\in A} \QQ_\alpha(k_\alpha)$, where $\QQ_\alpha$ are arbitrary parabolic $k_\alpha$-groups in $\GG_\alpha$. For any $I\se S$ we define the parabolic group
$$P_I\ =\ \prod_{\alpha\in A}\PP_{\alpha, I\cap S_\alpha}(k_\alpha)$$
and write $P=P_\vn$. The collection of these $P_I$ coincide with those parabolics that contain $P$; we call them \emph{standard} parabolics. The correspondence $I\mapsto P_I$ is an isomorphism of ordered sets. Any parabolic $Q$ can be conjugated to $P_I$ for some $I=I(Q)$ and the set of such $Q$ for a given $I$ identifies with $G/P_I$. We call $I(Q)$ the \emph{type} of $Q$. The relative Bruhat decomposition of the factors $\GG_\alpha$ (see Th\'eor\`eme~5.15 in~\cite{Borel-Tits65} and~\cite[\S3]{Borel-Tits72}) provides a decomposition
\begin{equation}\label{eq:Bruhat}
G/P = \bigsqcup_{w\in W} C(w), \kern1cm\text{where}\kern3mm C(w)= PwP /P.  
\end{equation}
\emph{Warning}: The notation~\eqref{eq:Bruhat} is consistent with~\cite{Borel-Serre76} (for a single field) since we shall use arguments analogous to theirs; however, in~\cite{Borel-Tits65,Borel-Tits72}, $C(w)$ would correspond to the pre-image $PwP\se G$.

\bigskip

The above setting gives rise to a Tits system and the associated Tits building $\ST$; namely, the Tits system is the product of the classical Tits systems of the factors, and $\ST$ the join of the Tits buildings of all $\GG_\alpha$. We now describe in more detail the Tits building $\ST$ viewed as a sot complex. For each $I\se S$, consider the compact $G$-space $G/P_I$, recalling that each $\GG_\alpha(k_\alpha)/\PP_{\alpha, I\cap S_\alpha}(k_\alpha)$ can be identified with the $k_\alpha$-points of the smooth projective $k_\alpha$-variety $\GG_\alpha/\PP_{\alpha, I\cap S_\alpha}$. For any $0\leq n\leq r-1$ we define the compact space
$$\ST^{(n)}\ =\ \bigsqcup_{|I|=r-1-n} G/P_I.$$
We note that this definition extends consistently to the singleton $\ST^{(-1)}=G/G$ but we shall not use this ``augmented'' simplicial complex in this section. One can consider $\ST^{(n)}$ as a space of $n$-simplices over the space $\ST^{(0)}$ of maximal proper parabolic subgroups $Q\subsetneq G$ as follows. A set $\{Q_0, \ldots, Q_n\}$ (distinct $Q_i$) is a simplex if and only if $Q=Q_0\cap\ldots \cap Q_n$ is parabolic; in that case, we have $Q\in\ST^{(n)}$ and we identify $Q$ with that simplex. More precisely, the type $I(Q)$ is $I(Q_0)\cap\ldots\cap I(Q_n)$.

\medskip

We can endow $\ST$ with a \emph{sufficient orientation} by means of the partition of $\ST^{(0)}$ into its $r$ types. More precisely, we fix any order on $S$ and for $x,x'\in\ST^{(0)}$ we declare $x<x'$ whenever $x\in G/P_{S\sm \{s\}}$, $x'\in G/P_{S\sm \{s'\}}$ with $s<s'$. Notice that all face maps are unions of canonical projections $G/P_I\to G/P_J$, where $I\se J$ and $|J|=|I|+1$.

\begin{theorem}\label{thm:complexe}
For every locally convex topological vector space $V$,
$$\hc^n(\ST;V)=0\kern5mm (\forall\,n\neq 0, r-1),\kern5mm\text{and}\kern3mm\hc^0(\ST;V)=V.$$
\end{theorem}

\begin{remark}\label{rem:complexe}
The above theorem gives a complete description of $\hc^\bu(\ST;V)$, because the remaining term of degree $r-1$ is readily understood. Indeed, since all $r$ face map $\partial_{r-1, j}$ on $\ST^{(r-1)}$ have \emph{disjoint} ranges, namely the various $G/P_{\{s\}}$, one has an identification
\begin{equation}\label{eq:top_degree}
\hc^{r-1}(\ST;V)\ =\ \cont(G/P;V)\Big/\textstyle\sum\limits_{s\in S}\cont(G/P_{\{s\}};V).
\end{equation}
Whilst~\eqref{eq:top_degree} keeps a clear view of the $G$-representation, another identification has the benefit of presenting $\hc^{r-1}$ as a function space with no quotient taken. To this end, let $w_0$ be the longest element of $W$. Since $C(w_0)$ is open in the compact space $G/P$ (see~\cite[4.2]{Borel-Tits65} or~3.13 and~3.15 in~\cite{Borel-Tits72}), there is a natural inclusion of the space $\cont_0(C(w_0);V)$ of functions vanishing at infinity into $\cont(G/P;V)$. We shall see in the proof of Theorem~\ref{thm:complexe} that this inclusion induces an isomorphism
$$\cont_0(C(w_0);V) \xrightarrow{\ \cong\ }\hc^{r-1}(\ST;V).$$
This identification can also be obtained by direct means.
\end{remark}

In preparation for the proof, we introduce a filtration along~\eqref{eq:Bruhat} which parallels simplicially for $\ST$ the topological filtration considered by Borel--Serre~\cite{Borel-Serre76} for the topological realization of the building associated to a semi-simple group. Let $\ell$ be the length function on $W$ with respect to $S$ and fix an enumeration $W=\{w_1, \ldots, w_N\}$ with the property that $\ell(w_i)\leq \ell(w_j)$ for all $i<j$; in particular, $w_1$ is the neutral element and $w_N=w_0$ the longest. For $1\leq m\leq N$, define a subcomplex $\ST_m\se \ST$ by retaining retaining only the subset
$$\bigcup_{j=1}^m Pw_j P_I/P_I \se G/P_I$$
as $I\subsetneq S$. In other words, $\ST_m$ is the subcomplex spanned by the $m$ collections $C(w_j)$ of chambers, $j\leq m$. Thus this defines an increasing sequence of subcomplexes of $\ST$. For instance, $\ST_1$ is the Weyl chamber at infinity and is isomorphic to the finite simplex of all proper standard parabolic subgroups, with dual inclusion; at the other extreme, $\ST_N=\ST$. It follows from Th\'eor\`eme~3.13 and Corollaire~3.15 in~\cite{Borel-Tits72} applied to each factor $\GG_\alpha$ that all $\ST_m$ are closed subcomplexes of $\ST$.

\begin{proposition}\label{prop:van_rel}
We have $\hc^\bu(\ST_m, \ST_{m-1};V)=0$ for every $1<m<N$ and any locally convex topological vector space $V$.
\end{proposition}

\begin{proof}
We define $I_m=\{s\in S : \ell(w_m s)> \ell(w_m)\}$ (as in Borel--Serre~\cite{Borel-Serre76}). The assumption $m\neq 1,N$ implies respectively $I_m\neq S, \vn$; see~\cite[IV\S1 Ex.~3]{BourbakiLIE456}. The type $I_m$ contains information about $\ST_m\sm \ST_{m-1}$ as follows: For any $I\nsubseteq I_m$, there are no simplices of type $I$ in $\ST_m\sm \ST_{m-1}$. On the other hand, if $I\se I_m$, then the map $G/P\to G/P_I$ restricts to a homeomorphism of $C(w_m)$ onto its image. (Both statements follow from~\cite[3.16]{Borel-Tits72} applied to each $\GG_\alpha$, see~\cite[2.4]{Borel-Serre76}.)

Consider now the $(r-1)$-simplex $\SC$ of all proper subsets $I\subsetneq S$ with \emph{dual} inclusion; thus its vertices are all the sets $S\sm \{s\}$ as $s$ ranges over $S$. This is a model for the Weyl chamber at infinity in which $\SC_I =\{J: I\se J \}$ is the face (subcomplex) fixed by all reflections $s\in I$. There is a canonical identification $\SC\cong \ST_1$, $I\mapsto P_I$; we orient $\SC$ accordingly. Consider further the subcomplex
\begin{equation}\label{eq:D_m}
\SD_m=\bigcup_{s\in S\sm I_m}\SC_{\{s\}} = \bigcup_{I\nsubseteq I_m}\SC_I \se\SC.
\end{equation}
Since $I_m\neq \vn, S$, the abstract simplicial complex $\SD_m$ is (non-empty and) contractible, being a union of a proper subset of codimension one faces. Following Example~\ref{exo:product}, we deduce that $\hc^n(\SD_m\times C(w_m); V)$ vanishes for $n>0$ and is (canonically) $V$ for $n=0$. The same statement holds for the sot complex $\SC\times C(w_m)$. In particular, Lemma~\ref{lemma:relative} implies
\begin{equation}\label{eq:rel_vanishing}
\hc^\bu(\SC\times C(w_m),\SD_m\times C(w_m); V)=0.
\end{equation}
We consider the map
$$\SC\times C(w_m) \lra \ST,\kern1cm (I, b w_m P)\longmapsto b w_m P_I/P_I \in G/P_I.$$
This map ranges in $\ST_m$ and preserves the sufficient orientations; it is a morphism and is open. The properties of $I_m$ mentionned above and~\eqref{eq:D_m} show that this map sends $\SD_m\times C(w_m)$ to $\ST_{m-1}$ and that it restricts to a homeomorphism
$$(\SC\sm \SD_m) \times C(w_m) \xrightarrow{\ \cong\ } \bigcup_{I\se I_m} P w_m P_I/P_I = \ST_m\sm \ST_{m-1}.$$
%
%
The claim of the proposition thus follows from Lemma~\ref{lemma:ex} and equation~\eqref{eq:rel_vanishing}.
\end{proof}

\begin{proof}[Proof of Theorem~\ref{thm:complexe}]
Apply inductively Proposition~\ref{prop:van_rel} to the long exact sequence of Lemma~\ref{lemma:relative}, starting with the finite simplex $\ST_1$ as subcomplex of $\ST_2$. Since the cohomology of $\ST_1$ is $V$ in degree zero and vanishes otherwise, the same is true for $\ST_{N-1}$. Since $\ST_N\sm\ST_{N-1} = C(w_N)$ consists of top-dimensional simplices only (specifically, all chambers opposite the chamber fixed by $P$), we have $\hc^n(\ST, \ST_{N-1};V)=0$ for all $n\neq r-1$. It follows that $\hc^n(\ST;V)$ vanishes for $n\neq 0, r-1$ while for $n=0$ it is $V$ and
$$\hc^{r-1}(\ST;V)\ \cong\ \hc^{r-1}(\ST,\ST_{N-1};V)\ =\ \cont(\ST^{(r-1)},\ST^{(r-1)}_{N-1};V).$$
The right hand side is $\cont(G/P,G/P\sm C(w_N);V)$. Since $C(w_N)$ is open in the compact space $G/P$, the restriction to $C(w_N)$ yields an identification
$$\cont(G/P,G/P\sm C(w_N);V)\ =\ \cont_0(C(w_N))$$
as claimed in Remark~\ref{rem:complexe}.
\end{proof}

\section{Complements on bounded cohomology}
\subsection{Background}
(For more details or proofs of the facts below, we refer to~\cite{Monod}.) Let $G$ be a locally compact group. A \emph{Banach $G$-module} $V$ is a Banach space $V$ with a linear representation of $G$ by isometries. We usually denote all such representations by $\pi$ (in contempt of the resulting abuses of notation). The module is called \emph{continuous} if the map $G\times V\to V$ is continuous, which is equivalent to the continuity of all orbit maps $G\to V$. A \emph{coefficient $G$-module} is the dual of a separable continuous $G$-module; it is in general neither separable nor continuous.

\begin{lemma}\label{lemma:cont_module}
Let $G$ be a locally compact group and $V$ a Banach $G$-module.
\begin{itemize}
\item[(i)] If all orbit maps are weakly continuous, then $V$ is continuous.
\item[(ii)] If $V$ is a separable coefficient module, then it is continuous.
\item[(iii)] If $V$ is the dual of a Banach $G$-module $V^\flat$, then all orbit maps of $V$ are {weak-$*$} continuous if and only if $V^\flat$ is continuous.
\end{itemize}
\end{lemma}

\begin{proof}
The first assertion is a classical fact, see \emph{e.g.}~\cite[2.8]{Glicksberg-Leeuw}. The second can be found in~\cite[3.3.2]{Monod}. The third follows from~(i).
\end{proof}

Let $\Omega$ be a standard ($\sigma$-finite) measure space with a measurable measure class preserving $G$-action. For any coefficient module $V$, the space $\lw(\Omega;V)$ of bounded weak-$*$ measurable function classes is a coefficient module when endowed with the representation
\begin{equation}\label{eq:regular}
\Big(\pi(g)f\Big)(\omega) = \pi(g)\Big(f(g^{-1}\omega)\Big).
\end{equation}
More generally, one defined an (adjoint) operator $\pi(\psi)$ for any $\psi\in L^1(G)$ by integrating~\eqref{eq:regular} in the sense of the Gelfand--Dunford integral~\cite[VI\S1]{BourbakiINT6}. We reserve the notation $L^p(\Omega;V)$, $p\leq \infty$, for strongly measurable $L^p$-maps (in the sense of Bochner) and write $L^p(\Omega)$ if $V=\RR$.

\medskip

R.~Zimmer's notion of \emph{amenability}~\cite[ch.~4]{Zimmer84} for the $G$-space $\Omega$ is equivalent to the appropriate concept of (relative) \emph{injectivity} for the modules $\lw(\Omega;V)$ and in particular suitable resolutions consisting of such modules realize the bounded cohomology $\hb^\bu(G;V)$.

\smallskip

Recall that there is a long exact sequence in bounded cohomology associated to suitable short exact sequences of coefficient modules. The natural setting in the context of relative homological algebra is to consider sequences that split in the category of Banach spaces and yields the expected long exact sequence in complete generality~\cite[8.2.7]{Monod}. However, using E.~Michael's selection theorem~\cite[7.2]{Michael56}, one can prove the long exact sequence for arbitrary short exact sequences of \emph{continuous} modules for locally compact groups~\cite[8.2.1(i)]{Monod}. A different argument also establishes the sequence for \emph{dual} short exact sequences of coefficient modules~\cite[8.2.1(ii)]{Monod}. (We point out however that the latter can also be reduced to the former thanks to a special case of Proposition~\ref{prop:cont_exact} below.)

\smallskip

In any case, successive applications of the long exact sequence yield the following result.

\begin{proposition}\label{prop:ex_sequence}
Let $G$ be a locally compact second countable group, $r\geq 1$ and let
$$0 \lra  V=W_{-1} \xrightarrow{\ d_0\ } W_0 \xrightarrow{\ d_1\ }\cdots \xrightarrow{\ d_{r}\ } W_{r} \xrightarrow{\ d_{r+1}} 0$$
be an exact sequence of either dual morphisms of coefficient $G$-modules or morphisms of continuous Banach $G$-modules. Suppose that $\hb^q(G;W_p)$ vanishes for all $p\geq 0$, $q\geq 0$ with $p+q\leq r-1$. Then
$$\hb^n(G;V)=0 \kern10mm \forall\, 0\leq n \leq r-1.$$
\end{proposition}

\begin{remark}\label{rem:ex_sequence}
Notice that the vanishing assumptions do not concern $W_r.$ Thus we could equivalently work with an exact sequence terminating with $W_{r-2}\to W_{r-1}$ and the additional assumption that the latter map has closed range, setting $W_r=W_{r-1}/dW_{r-2}$, which in the dual case is a coefficient module by the closed range theorem.
\end{remark}

\begin{proof}[Proof of Proposition~\ref{prop:ex_sequence}]
We shall prove the following statement by induction on $q$ from $q=0$ to $q= r-1$:
\begin{equation}\label{eq:rec}
\hb^q(G; \ker d_{p+1})\ =\ 0 \kern10mm \forall\,p\geq 0 \text{ with } p+q\leq r-1.
\end{equation}
This yields the statement of the proposition since at $p=0$ we have $\ker d_1\cong V$. We observe that all spaces $\ker d_\bu$ have a natural structure of Banach (respectively coefficient) $G$-module since $\ker d_\bu$ is closed (respectively {weak-$*$} closed). For $q=0$, we have $(\ker d_{p+1})^G\se W_p^G$ which vanishes, yielding~\eqref{eq:rec}. Now if $q\geq 1$, we have a short exact sequence
$$0 \lra \ker d_{p+1} \lra W_p \lra \ker d_{p+2}\lra 0$$
which, in the dual case, is still dual. The associated long exact sequence~\cite[8.2.1]{Monod} contains the following piece:
$$\hb^{q-1}(G; \ker d_{p+2})\lra \hb^q(G; \ker d_{p+1})\lra \hb^q(G; W_p).$$
The induction hypothesis together with the vanishing assumptions entails the vanishing of $\hb^q(G; \ker d_{p+1})$. This is the induction step and thus concludes the proof.
\end{proof}

\noindent
(A similar but less explicit proof follows from considering the spectral sequence that $\hb^\bu(G;-)$ associates to the exact sequence in the statement.)

\subsection{Complements on Banach modules}\label{sec:complements_modules}
If $G$ is a topological group and $V$ is any Banach $G$-module, we denote by $\cm V$ (or $\cm_G V$ when necessary) the collection of all elements $v\in V$ for which the associated orbit map $G\to V$ is continuous. This defines a closed invariant subspace and any Banach $G$-module morphism $V\to U$ restricts to $\cm V\to\cm U$. Moreover, the definition of bounded cohomology implies readily
\begin{equation}\label{eq:hb_cm}
\hb^\bu(G;\cm V) = \hb^\bu(G;V).
\end{equation}
(See~\S1.2 and~6.1.5 in~\cite{Monod} for this and more on $\cm$.) This does not, however, allow us to deal exclusively with continuous modules, because weak-$*$ limiting operations involved in amenability require duality, which is not preserved under the functor~$\cm$.

\begin{proposition}\label{prop:cont_module}
Let $G$ be a locally compact second countable group and $V$ a coefficient $G$-module. Then $\cm V = L^1(G)V = L^1(G) \cm V$.
\end{proposition}

\begin{proof}
We sketch the argument given in the proof of~\cite[3.2.3]{Monod}. The inclusion $L^1(G)V\se \cm V$ follows from the continuity of the $G$-action on $L^1(G)$. The existence of approximate identities in $L^1(G)$ (see \emph{e.g.}~\cite[13.4]{Doran-Wichmann}) shows that $L^1(G) \cm V$ is dense in $\cm V$. The factorization theorem of J.P.~Cohen (see~\cite[16.1]{Doran-Wichmann}) implies that $L^1(G) \cm V$ is closed in $\cm V$. Thus $\cm V = L^1(G) \cm V \se L^1(G)V\se \cm V$, completing the proof.
\end{proof}

Let $0\to A \xrightarrow{\alpha} B \xrightarrow{\beta} C \to 0$ be an exact sequence of (morphisms of) Banach modules. Then $\alpha$ has closed range and thus admits a continuous inverse by the open mapping theorem; it follows readily that one has
$$0\lra \cm A \xrightarrow{\ \alpha\ } \cm B \xrightarrow{\ \beta\ } \cm C \lra 0$$
exact except possibly at the right. It turns out that the functor $\cm$ is actually exact on dual exact sequences of coefficient modules and that there is also a converse statement:

\begin{proposition}\label{prop:cont_exact}
Let $G$ be a locally compact second countable group and $d_n : L_{n-1}\to L_n$ dual morphisms of coefficient $G$-modules $(n=1,2)$. Then
$$L_0 \xrightarrow{\ d_1\ } L_1 \xrightarrow{\ d_2\ } L_2$$
is an exact sequence if and only if
$$\cm L_0 \xrightarrow{\ d_1\ } \cm L_1 \xrightarrow{\ d_2\ } \cm L_2$$
is exact.
\end{proposition}

\begin{proof}
\emph{Sufficiency.} Recall first that for any coefficient $G$-module $V$, the closed submodule $\cm V$ is {weak-$*$} dense in $V$. Indeed, if $\{\psi_\alpha\}$ is a net of functions $\psi_\alpha\in L^1(G)$ that constitutes a suitable approximate identity, then $\psi_\alpha v$ is in $\cm V$ and tends {weak-$*$} to $v$, see~\cite[3.2.3]{Monod}. It follows in particular that $d_2\,d_1=0$ since this map is {weak-$*$} continuous and vanishes on $\cm L_0$.

Let now $v\in\ker(d_2)$. The net $\{\pi(\psi_\alpha) v\}$ converges to $v$ and belongs to $\cm L_1\cap \ker(d_2)$ since $d_2$ is dual and hence commutes with $\pi(\psi_\alpha)$. Therefore, there is for each $\alpha$ some $u_\alpha\in \cm L_0$ with $d_1(u_\alpha)=\pi(\psi_\alpha) v$. Notice that the norm of $\pi(\psi_\alpha) v$ is bounded by $\|v\|$. Therefore, by the open mapping theorem applied to the surjective map $\cm L_0 \to \cm\!\ker(d_2)$ obtained by restricting $d_1$, we can choose the family $\{u_\alpha\}$ to be bounded in norm. By the Banach--Alao\u{g}lu theorem, the net $\{u_\alpha\}$ has a {weak-$*$} accumulation point $u\in L_0$. Since $d_1$ is dual, $d_1(u)=v$.

\emph{Necessity.} We shall prove more, assuming merely that there be closed submodules $C_n\se L_n$ containing $\cm L_n$ and such that the map $d_n$ restrict to an exact sequence $C_0\to C_1\to C_2$. Let now $v$ be in the kernel $\ker(d_2|_{\cm L_1})$. Observe that the latter coincides with $\cm\!\ker(d_2)$ and that $\ker(d_2)$ is a coefficient module since $d_2$ is {weak-$*$} continuous. Therefore, by Proposition~\ref{prop:cont_module} there is $v'\in \ker(d_2|_{\cm L_1})$ and $\psi\in L^1(G)$ with $v=\pi(\psi) v'$. Since $v'\in \cm L_1\se C_1$, there is $w'\in C_0\se L_0$ with $d_1(w')=v'$. Define $w=\pi(\psi) w'$, which is in $\cm L_0$. Since $d_0$ is dual, it commutes with $\pi(\psi)$ and therefore $v=d_0(w)$.
\end{proof}

\begin{lemma}\label{lemma:cont_lw}
Let $G$ be a locally compact second countable group and $V$ the dual of a separable Banach space. Consider the coefficient $G$-module $\lw(G;V)$ where $V$ is endowed with the trivial $G$-action. Then $\cm\lw(G;V)$ coincides with the space of bounded right uniformly continuous functions $G\to V$.
\end{lemma}

\begin{proof}
 We sketch the argument given in~\cite[4.4.3]{Monod}. By the above Proposition~\ref{prop:cont_module}, any continuous vector $f$ can be written $\pi(\psi) f'$ for some $\psi\in L^1(G)$ and $f'\in \lw(G;V)$. The claim then follows readily by continuity of the $G$-representation on $L^1(G)$.
\end{proof}

\begin{corollary}\label{cor:cont_lw}
Let $G$ be a locally compact second countable group, $H<G$ a closed cocompact subgroup and $V$ the dual of a separable Banach space. Then
$$\cm \lw(G/H;V)\ =\ \cont(G/H;V).\eqno{\text\qedsymbol}$$
\end{corollary}

\begin{remark}\label{rem:cont_lw}
If $V$ is a coefficient $G$-module with non-trivial action, then the above statements are \emph{a priori} no longer true, due to the lack of continuity of the representation on $V$. One can however establish another statement in that setting. Write $\cw(G/H;V)$ for the space of {weak-$*$} continuous maps, which are automatically bounded since {weak-$*$} compact sets are norm-bounded by the Banach--Steinhaus theorem~\cite{Banach-Steinhaus}. Then $\cm \lw(G/H;V)$ is contained in $\cw(G/H;V)$ and therefore coincides with $\cm\cw(G/H;V)$. (Notice that $\cw(G/H;V)$ is closed in $\lw(G/H;V)$ and hence is a Banach $G$-module.) This statement is proved by considering the isometric involution $A$ of $\lw(G;V)$ defined by $(Af)(g)= \pi(g)[ f(g^{-1})]$ and applying Lemma~\ref{lemma:cont_lw}.
\end{remark}

For the next theorem, we use the notations of Section~\ref{sec:tits}. We recall that the Haar measures induce a canonical invariant measure class on homogeneous spaces (Theorem~23.8.1 in~\cite{Simonnet}). Accordingly, there is a canonical measure class on each $\ST^{(p)}$. The following is a measurable analogue of Theorem~\ref{thm:complexe}.

\begin{theorem}\label{thm:complexe:lw}
Let $V$ be the dual of a separable Banach space. We have an exact sequence
\begin{equation}\label{eq:complexe:lw}
0\lra V\lra \lw(\ST^{(0)}; V) \lra \lw(\ST^{(1)}; V) \lra \cdots\lra \lw(\ST^{(r-1)}; V)
\end{equation}
of dual morphisms of coefficient $G$-modules with closed ranges.
\end{theorem}

\begin{proof}
The closed range condition needs only to be proved for the last map since elsewhere it will follow from exactness. By the closed range theorem, it is enough to show that the pre-dual map $L^1(\ST^{(r-1)};V^\flat) \to L^1(\ST^{(r-2)};V^\flat)$ has closed range, equivalently is open. But this map is the direct sum over $s\in S$ of the $r$ surjections $L^1(G/P;V^\flat) \to L^1(G/P_{\{s\}};V^\flat)$, each of which is open by the open mapping theorem.

We now turn to exactness. We apply Theorem~\ref{thm:complexe} and have an exact sequence
$$0\lra V\lra \cont(\ST^{(0)}; V) \lra \cont(\ST^{(1)}; V) \lra \cdots\lra \cont(\ST^{(r-1)}; V).$$
Each space $\cont(\ST^{(p)}; V)$ is a direct sum of terms $\cont(G/P_I; V)$, and therefore Corollary~\ref{cor:cont_lw} shows that $\cont(\ST^{(p)}; V)$ coincides with $\cm\lw(\ST^{(p)}; V)$. As for the initial term, $V=\cm V$ since we chose the trivial $G$-action on $V$. Therefore Proposition~\ref{prop:cont_exact} implies the exactness of the sequence~\eqref{eq:complexe:lw}.
\end{proof}

\begin{remark}\label{rem:take_advantage}
We take advantage of the fact that the morphisms in the sequence~\eqref{eq:complexe:lw} do not depend on the module structure on $V$. However, one can also establish the exactness of~\eqref{eq:complexe:lw} differently when $V$ is a non-trivial coefficient $G$-module. In that case, apply Theorem~\ref{thm:complexe} to the locally convex space given by $V$ in its {weak-$*$} topology to obtain an exact sequence of modules $\cw(\ST^{(p)}; V)$ starting with $V$. The stronger version of the ``necessity'' part that we proved for Proposition~\ref{prop:cont_exact} shows that the sequence $\cm\cw(\ST^{(p)}; V)$ starting with $\cm V$ is also exact. Since $\cm\cw(\ST^{(p)}; V)=\cm\lw(\ST^{(p)}; V)$ by Remark~\ref{rem:cont_lw}, one concludes again with Proposition~\ref{prop:cont_exact}.
\end{remark}

\subsection{Semi-separable modules and Mautner phenomenon}\label{sec:semi-separable}
\begin{definition}\label{defi:semi-separable}
Let $G$ be a locally compact second countable group and $V$ a coefficient $G$-module. We call $V$ \emph{semi-separable} if there exists a separable coefficient $G$-module $U$ and an injective dual $G$-morphism $V\to U$.
\end{definition}

\begin{remark}\label{rem:semi-separable}
For Proposition~\ref{prop:mautner} below, we shall only use that the coefficient $G$-module $U$ is continuous, which follows from Lemma~\ref{lemma:cont_module}. We will not need that the map $V\to U$ is adjoint, continuous or even linear; the above definition could therefore be replaced with requiring that $V$ admits an auxiliary uniform structure for which the action is continuous and by uniformly equicontinuous uniform maps. This would however prevent further applications of this concept, such as in Section~\ref{sec:ergodic}.
\end{remark}

Semi-separability is not a restriction on the underlying Banach space, since the dual of any separable Banach space admits an injective dual continuous linear map into a separable dual, indeed even into $\ell^2$. It is however a restriction on the $G$-representation; for instance, let $G$ be a countable group of homeomorphisms of some compact metrisable space $K$. Let $V$ be the $G$-module of Radon measures on $K$ with integral zero, which is a coefficient module. One can verify that $V$ is not semi-separable when $K$ does not admit a $G$-invariant measure.

\medskip

For a basic non-separable semi-separable example, let $\Omega$ be a standard probability space with a measurable measure-preserving $G$-action. The coefficient module $\lft(\Omega)$ is non-separable (unless $\Omega$ is atomic and finite). It is however semi-separable in view of the map $\lft(\Omega)\to L^2(\Omega)$, which is dual.

\smallskip

Increasing the generality, let $H$ be another locally compact second countable group and $\beta:G\times \Omega\to H$ a measurable cocycle. Let $V_0$ be a coefficient $H$-module and define a coefficient $G$-module $V=\lw(\Omega;V_0)$ by the (dual) representation
\begin{equation}
\label{eq:twisted}
(\pi(g) f)(\omega)\ =\ \pi(\beta(g^{-1},\omega)^{-1}) f(g^{-1}\omega),
\end{equation}
where $g\in G$ and $\omega\in\Omega$.

\begin{lemma}\label{lemma:semi-separable}
If $V_0$ is semi-separable, then so is $V$.
\end{lemma}

\begin{proof}
Let $U_0$ be as in Definition~\ref{defi:semi-separable} for $V_0$ and define $U=L^2(\Omega; U_0)$ with representation given by the same formula~\eqref{eq:twisted}. Since $U_0$ is a separable dual (of, say, $U_0^\flat$), it has the Radon--Nikod\'ym property~\cite[III.3.2]{Diestel-Uhl} and hence $U$ is the dual of $L^2(\Omega; U_0^\flat)$, see~\cite[IV.1.1]{Diestel-Uhl}. Now the map $V_0\to U_0$ induces an adjoint injective $G$-map
$$V=\lw(\Omega;V_0) \lra \lw(\Omega;U_0) = \lft(\Omega;U_0) \lra L^2(\Omega; U_0) =U$$
to a separable coefficient module.
\end{proof}

We provide now an example showing at once that the existence of a finite invariant measure is crucial in Lemma~\ref{lemma:semi-separable} and that the main results of this paper do not hold in the absence of semi-separability.

\begin{example}\label{ex:contre-ex}
Let $G=\GG(\RR)$ be a simple Lie group and assume that the associated symmetric space is of Hermitian type. Assume that $\GG$ is connected and simply connected (as an algebraic group, not as a Lie group). For instance, consider the symplectic groups $G=\SP_{2n}(\RR)$. Then $\hb²(G;\RR)$ is one-dimensional~\cite[\S5.3]{Burger-Monod3}. Let now $V=\lft(G)/\RR$; this is a coefficient $G$-module since it is the dual of the continuous separable $G$-module $L^1_0(G)$ of integral zero $L^1$-functions. The dimension-shifting trick~\cite[10.3.5]{Monod}, which is just an application of the long exact sequence, shows $\hb^2(G;\RR)\cong \hb^1(G;V)$. Thus $\hb^1(G;V)$ is non-trivial regardless of the rank of $G$, which is $n$ in the example at hand. This shows the necessity of the semi-separability assumption since $V^G=0$.

We develop the example a bit further to show that one gets counter-examples even without involving in any way trivial modules (as here $\RR$). Consider the case $n=1$ of the previous example, so that $G\cong \SL_2(\RR)$. Let $V_0$ be any (non-trivial) irreducible continuous unitary $G$-representation of spherical type and set $V=\lft(G;V_0)/V_0$. It is known that $\hb^2(G;V_0)$ is one-dimensional~\cite[1.1]{Burger-MonodERN}, and thus as before $\hb^1(G;V)$ is non-trivial. Again, this would be incompatible with the statement of Theorem~\ref{thm:simplesimple} but for the semi-separability assumption. These examples also show the necessity of semi-separability in Theorem~\ref{thm:splitting}, setting $\ell=1$.
\end{example}

We now adopt the notations introduced at the beginning of Section~\ref{sec:tits}. For continuous unitary representations, the following is the well-known Mautner phenomenon~\cite[II.3]{Margulis}.

\begin{proposition}\label{prop:mautner}
Let $V$ be a semi-separable coefficient $G$-module, $I\se S$ and $v\in V^{T_I}$. Then $v$ is fixed by some almost $k_\alpha$-simple factor of $\GG_\alpha(k_\alpha)$ for all $\alpha\in A$ such that $S_\alpha\not\subseteq I$.
\end{proposition}

\begin{remark}
The set $S_\alpha \sm I$ determines which almost $k_\alpha$-simple factor(s) will fix $v$. In fact, being simply connected, the $\GG_\alpha$ are direct products of their almost simple factors (3.1.2 p.~46 in~\cite{Tits66}. Thus, upon adding multiplicities to $A$ accordingly, one can get a more precise statement for products of almost simple groups.
\end{remark}

\begin{proof}[Proof of Proposition~\ref{prop:mautner}]
Let $V\to U$ be the map of Definition~\ref{defi:semi-separable}. Since it is injective, it suffices to prove the proposition for the image of $v$ in the $G$-module $U$, which is continuous (Remark~\ref{rem:semi-separable}). Therefore, we suppose without loss of generality that $V$ itself is continuous.

The classical Mautner lemma as established in~\cite[II.3.3(b)]{Margulis} now finishes the proof. Indeed, the entire argument given therein applies in greater generality and really shows the following. If $G$ acts continuously by isometries on some metric space $V$ and $v\in V$ is $T_I$-fixed, then it fixed by some almost $k_\alpha$-simple factor of $\GG_\alpha(k_\alpha)$ for all $\alpha\in A$ such that $S_\alpha\not\subseteq I$. For the reader's convenience, we sketch the argument. First, the continuity and isometry assumptions imply immediately that $v$ is fixed by any $g\in G$ such that the closure of $\{t g t^{-1} : t\in T_I\}$ contains the identity. Next, the contracting properties of the $T_I$-action on the unipotent radical of $P_I$ determine a large part of this radical consisting of such $g$. The same argument holds for the radical of the opposite parabolic $P_I^{-}$ and the Bruhat decomposition implies that together these unipotents generate the adequate almost simple factor(s).
\end{proof}

It follows from the definition that a {weak-$*$} closed submodule of a semi-separable coefficient module is still a semi-separable coefficient module; we record for later use that the property passes also to quotients in a special case:

\begin{lemma}\label{lemma:semi-separable_quotient}
Let $G$ be a locally compact second countable group, $N\lhd G$ a normal subgroup and $V$ a semi-separable coefficient $G$-module. Then $V/V^N$ is a semi-separable coefficient $G$-module.
\end{lemma}

\begin{proof}
Let $V\to U$ be as in Definition~\ref{defi:semi-separable}. Observe that $V^N$, $U^N$ are {weak-$*$} closed and $G$-invariant; hence $V/V^N$, $U/U^N$ are also coefficient $G$-modules. The $G$-equivariance of the injection $V\to U$ implies the injectivity (and definiteness) of the map $V/V^N \to U/U^N$, which is also adjoint.
\end{proof}

\section{Vanishing}
\subsection{Semi-simple groups}\label{sec:vanish_ss}
Theorem~\ref{thm:simplesimple} is a special case of (both statements in) Theorem~\ref{thm:semi-simple} and will therefore not be discussed separately.

\begin{remark}[{{\bfseries compact factors}}]\label{rem:compact_factors}
Let $G$ be as for Theorem~\ref{thm:semi-simple} and let $K\lhd G$ be the product of the groups of $k_\alpha$-points of all $k_\alpha$-anisotropic factors of all $\GG_\alpha$, which is compact (see \emph{e.g.}~\cite{Prasad82}). We thus have canonical identifications
$$\hb^\bu(G; V)\ \cong\ \hb^\bu(G; V^K)\ \cong\ \hb^\bu(G/K; V^K)$$
for any Banach $G$-module $V$, see~\cite[8.5.6, 8.5.7]{Monod}. Since $k_\alpha$-anisotropy is equivalent to the vanishing of the $k_\alpha$-rank, we have $\rank(G/K)=\rank(G)$ and $\minrank(G/K)\geq\minrank(G)$. This implies that we can and shall assume throughout the proof of Theorem~\ref{thm:semi-simple} that there are no anisotropic factors. Point~(i) will then be established under the slightly weaker hypothesis that there are no vectors fixed simultaneously by some isotropic factor and by $K$.
\end{remark}

\begin{proof}[Proof of Theorem~\ref{thm:semi-simple}~(i)]
We adopt the notations of the theorem; since by Remark~\ref{rem:compact_factors} we can assume for all $\alpha$ that $\GG_\alpha$ has no $k_\alpha$-anisotropic factors, we are in the situation of Section~\ref{sec:tits}. However, we shall rather apply the results of that section to the group $\ti G=G\times G$, itself also an example of the same setting. In an attempt at suggestive notation (and perhaps with a view on potential extensions to Kac--Moody groups), we write $\ti A= A\sqcup A^-$, $\ti S= S\sqcup S^-$, etc., where $A^-$ is just a disjoint copy of $A$, and so on. We then choose opposite roots for the second factor, so that when $I\se \ti S$ is entirely contained in $S^-$, the group $P_I^-$ is indeed the parabolic opposite to the corresponding parabolic in the first factor, justifying the notation.

The rank $r$ of Section~\ref{sec:tits} is $r=2\,\rank(G)$. We apply Theorem~\ref{thm:complexe:lw} and obtain an exact sequence
\begin{equation}\label{eq:lw_tiG}
0\lra V\lra \lw(\ST^{(0)}; V) \lra \lw(\ST^{(1)}; V) \lra \cdots\lra \lw(\ST^{(r-1)}; V)
\end{equation}
of adjoint maps with closed ranges. We set $W_p=\lw(\ST^{(p)}; V)$ for $p\leq r-1$ and $W_r=W_{r-1}/dW_{r-2}$. Let $G$ act on each $\ST^{(p)}$ \emph{via} the diagonal embedding $\Delta: G\to \ti G$. Now each $W_p$ is a coefficient $G$-module for the representation~\eqref{eq:regular} which uses the $G$-representation on $V$. (This representation was not used in establishing the exactness of the sequence~\eqref{eq:lw_tiG} of $\ti G$-modules, and indeed $\ti G$ does not act on $V$.) In conclusion, we have a dual exact sequence
$$0 \lra  V=W_{-1} \xrightarrow{\ d_0\ } W_0 \xrightarrow{\ d_1\ }\cdots \xrightarrow{\ d_{r}\ } W_{r} \xrightarrow{\ d_{r+1}} 0$$
coefficient $G$-modules. In order to prove case~(i) of Theorem~\ref{thm:semi-simple}, it suffices now to check the vanishing assumptions of Proposition~\ref{prop:ex_sequence}.

We prove that in fact $\hb^\bu(G; W_p)$ vanishes altogether for all $0\leq p\leq r-1$. Since $\hb^\bu(G;-)$ commutes with finite sums~\cite[8.2.10]{Monod}, it suffices to show
$$\hb^\bu\big(G; \lw(\ti G/\ti P_{\ti I};V)\big)\ =\ 0$$
for all $\ti I$ with $|\ti I|=r-1-p$. We can write $\ti P_{\ti I}=P_{I}\times P_{I^-}^-$, with $P_I$ and $P_{I^-}$ standard parabolics for $G$ and $\ti I=I\sqcup I^-$. The set $P_I\cdot P_{I^-}^- \se G$ contains $P\cdot P^-$, which is of full measure. Indeed, since $P^-= w_0 P w_0$, the set $P\cdot P^-$ is the translated of the big cell in the Bruhat decomposition, which is Zariski-open (compare also~\cite[4.2]{Borel-Tits65}). Since the $G$-action on $\ti G$ is diagonal, it follows that there is a canonical isomorphism of coefficient $G$-modules
$$\lw(\ti G/\ti P_{\ti I};V)\ \cong\ \lw(G/(P_I\cap P_{I^-}^-);V)$$
(where now in the right hand side $P_{I^-}^-$ is also considered as a subgroup of the same factor $G$). We now apply the induction isomorphism (\emph{\`a la} Eckmann--Shapiro) for bounded cohomology~\cite[10.1.3]{Monod} to the subgroup $P_I\cap P_{I^-}^-$ of $G$, recalling that it takes a simpler form since $V$ is a $G$-module rather than just a $P_I\cap P_{I^-}^-$-module~\cite[10.1.2(v)]{Monod}; namely:
$$\hb^\bu\big(G;\lw(G/(P_I\cap P_{I^-}^-);V)\big)\ \cong\ \hb^\bu(P_I\cap P_{I^-}^-;V).$$
At this point we observe that the condition $p\geq 0$ forces at least one of the sets $I$, $I^-$ to be a proper subset of $S$. We shall assume $I\neq S$, the other case being dealt with in a symmetric fashion. By abuse of language, we call \emph{L\'evi decomposition} the decomposition $P_I= V_I \rtimes \SZ_G(T_I)$, where $\PP_{\alpha, I\cap S_\alpha}=\VV_{\alpha, I\cap S_\alpha} \rtimes \SZ_{\GG_\alpha}(\TT_{\alpha,I\cap S_\alpha})$ is a L\'evi decomposition and $V_I = \prod_{\alpha\in A}\VV_{\alpha, I\cap S_\alpha}(k_\alpha)$. Consider the subgroup
$$R\ =\ (V_I\rtimes T_I)\cap P_{I^-}^-\ <\ P_I\cap P_{I^-}^-$$
and observe that it is normal. Being soluble and hence amenable, there is an isomorphism~\cite[8.5.3]{Monod}
$$\hb^\bu(P_I\cap P_{I^-}^-;V))\ \cong\ \hb^\bu(P_I\cap P_{I^-}^-;V^R).$$
Since $T_I<R$, we have $V^R\se V^{T_I}$. Thus Proposition~\ref{prop:mautner} implies that the elements of $V^R$ are fixed by some almost $k_\alpha$-simple factor of $\GG_\alpha(k_\alpha)$ for all $\alpha\in A$ such that $S_\alpha\not\subseteq I$. There is at least one such $\alpha$ since $I\neq S$. In view of the assumption of Theorem~\ref{thm:semi-simple}, we conclude $V^R=0$; this completes the proof that $\hb^\bu(G; W_p)$ vanishes.
\end{proof}

\subsection{A first glance at products}
Our main results on products will be established in Section~\ref{sec:products} below; in order to complete the proof of Theorem~\ref{thm:semi-simple}~(ii), we present here a different type of argument where the number of factors does not contribute anything to the vanishing. Consider the following condition on a locally compact second countable group $H$ for an integer $m\in \NN$.

\begin{itemize}
\item[$(*_m)$] The space $\hb^q(H;V)$ vanishes for all semi-separable coefficient $H$-modules $V$ with $V^H=0$ and all $q\leq m$.
\end{itemize}

\begin{proposition}\label{prop:product_glance}
Let $G= G_1\times \cdots\times G_\ell$ be a product of locally compact second countable groups and let $m\in \NN$. If each $G_i$ has property~$(*_m)$, then so does $G$.
\end{proposition}

\begin{proof}
We argue by induction on $\ell$ and observe that the case $\ell=1$ is tautological. Let now $\ell\geq 2$. Upon writing $G$ as $G_1\times \prod_{i\neq 1}G_i$, we see that the induction hypothesis allows us to suppose $\ell=2$. We thus set $G=G_1\times G_2$ and consider a semi-separable coefficient $G$-module $V$ with $V^G=0$.

We claim first that $\hb^q(G;U)$ vanishes for $q\leq m$ for any semi-separable coefficient $G$-module $U$ with either $U^{G_1}=0$ or $U^{G_2}=0$. By symmetry, we consider only $U^{G_1}=0$. We consider the (dual) exact sequence of $G$-modules
$$0\lra U \lra \lw(G_2;U) \lra \lw(G_2^2;U) \lra \cdots \lra \lw(G_2^{p+1};U)\lra \cdots$$
In view of Proposition~\ref{prop:ex_sequence} and Remark~\ref{rem:ex_sequence}, it suffices to prove that $\hb^q(G;\lw(G_2^{p+1};U))$ vanishes for all $p\geq 0$ and all $q\leq m$. We recall the identification of $\lw(G_2^{p+1};U)$ with $\lw(G/G_1;\lw(G_2^p;U))$, see~\cite[2.3.3]{Monod}, and the special form of the induction isomorphism already mentioned above (10.1.2(v) and 10.1.3 from~\cite{Monod}); it follows the isomorphism
\begin{equation}\label{eq:isom_ind}
\hb^q(G;\lw(G_2^{p+1};U))\ \cong\ \hb^q\big(G_1; \lw(G_2^p;U)\big),
\end{equation}
wherein $G_1$ acts trivially on $G_2^p$, $p\geq 0$. In particular, we may represent the Haar measure class by a finite measure and we see that $\lw(G_2^p;U)$ is semi-separable \emph{as a $G_1$-module}. Now since $\lw(G_2^p;U)^{G_1}$ is $\lw(G_2^p;U^{G_1})$ which vanishes, the right hand side of~\eqref{eq:isom_ind} vanishes by the assumption on $G_1$. In conclusion, the claim is established indeed.

The space $V^{G_1}$ is {weak-$*$} closed and we have a dual exact sequence of coefficient $G$-modules
\begin{equation}\label{eq:exact_left-right}
0\lra V^{G_1} \lra V \lra V/V^{G_1} \lra 0.
\end{equation}
All three are semi-separable, see Lemma~\ref{lemma:semi-separable_quotient}. By assumption, we have $(V^{G_1})^{G_2}=0$. On the other hand, one verifies readily that $(V/V^{G_1})^{G_1}$ vanishes, see~\cite[1.2.10]{Monod}. Therefore, the long exact sequence~\cite[8.2.1(ii)]{Monod} associated to~\eqref{eq:exact_left-right}, together with the above claim, shows that $\hb^q(G;V)$ vanishes for all $q\leq m$.
\end{proof}

\begin{proof}[Proof of Theorem~\ref{thm:semi-simple}~(ii)]
Since each $\GG_\alpha$ is simply connected, it is the direct product of its almost simple factors (see~3.1.2 p.~46 in~\cite{Tits66}). Therefore, we can assume that each $\GG_\alpha$ is almost $k_\alpha$-simple. The statement now follows by combining the case of a single factor in~(i) with Proposition~\ref{prop:product_glance} for $m=2\,\minrank(G)-1$.
\end{proof}

\subsection{Lattices}\label{sec:lattices}
Let $G$ be a locally compact second countable group, $\Gamma<G$ a lattice and $W$ a coefficient $\Gamma$-module. Recall~\cite[10.1.1]{Monod} that one defines the $\lft$-\emph{induced} coefficient $G$-module $V=\lw(G;W)^\Gamma$ by the right translation $G$-action. This module can also be viewed as a special case of the construction given for ergodic-theoretical cocycles in Section~\ref{sec:semi-separable} by taking for $\Omega$ a Borel fundamental domain for $\Gamma$ in $G$ and considering the corresponding cocycle $\beta: G\times \Omega\to \Gamma$. In particular, Lemma~\ref{lemma:semi-separable} applies:

\begin{lemma}\label{lemma:ind:lattices}
If $W$ is semi-separable, then the $\lft$-induced module $V=\lw(G;W)^\Gamma$ is a semi-separable coefficient $G$-module.\hfill\qedsymbol
\end{lemma}

The induction isomorphism~\cite[10.1.3]{Monod} reads
\begin{equation}\label{eq:induction}
\hb^\bu(\Gamma;W)\ \cong \ \hb^\bu(G;V).
\end{equation}
The two corollaries not involving irreducibility can now be readily deduced from the corresponding theorems on semi-simple groups. Corollary~\ref{cor:lattice_simple} being a special case of Corollary~\ref{cor:lattice_semisimple_mink}, we give:

\begin{proof}[Proof of Corollary~\ref{cor:lattice_semisimple_mink}]
A $G$-invariant element of the induced semi-separable coefficient $G$-module $V$ is a constant map $G\to W$; being $\Gamma$-equivariant, it ranges in $W^\Gamma$ and thus vanishes. Therefore, point~(ii) of Theorem~\ref{thm:semi-simple} together with the induction isomorphism~\eqref{eq:induction} yields the statement of the corollary.
\end{proof}

We now introduce the additional material needed for the general results on irreducible lattices. Let again $G$, $\Gamma$ and $W$ be as in the beginning of this Section~\ref{sec:lattices}. Let further $N\lhd G$ be a closed normal subgroup. In keeping with the notation of the Introduction, denote by $W_N\se W$ the collection of all $w\in W$ such that the orbit map $\Gamma\to W$, $\gamma\mapsto \pi(\gamma)w$ can be extended to a continuous map $G\to W$ factoring through $G/N$. This set is possibly reduced to zero.

\begin{lemma}\label{lemma:fixed_induced}
If $\Gamma\cdot N$ is dense in $G$, then $W_N$ is a closed $\Gamma$-invariant linear subspace and thus $W_N$ has a structure of continuous $G$-module that extends the $\Gamma$-structure and descends to a continuous $G/N$-module structure. Moreover, there is a canonical $G$-equivariant isometric isomorphism $W_N\cong \cm V^N$.
\end{lemma}

\begin{remark}\label{rem:not_ambiguous}
The notation $\cm V^N$ is not ambiguous since one checks
$$(\cm_G V)^N = \cm_G(V^N) = \cm_{G/N}(V^N)$$
(notation introduced in Section~\ref{sec:complements_modules}).
\end{remark}

\begin{proof}[Proof of Lemma~\ref{lemma:fixed_induced}]
In view of Lemma~\ref{lemma:cont_lw} (with left and right exchanged), the elements of $\cm V$ are continuous maps and may thus be evaluated at the identity. This provides us with a $\Gamma$-equivariant map $\cm V\to W$. The density assumption implies that its restriction $E:\cm V^N\to W$ is isometrically injective; moreover, it ranges in $W_N$. On the other hand, for any $w\in W_N$ there is a continuous map $G\to W$, $g\mapsto \pi(g)w$ which is automatically left uniformly continuous. This provides an isometric right inverse to $E:\cm V^N\to W_N$.
\end{proof}

In fact we shall often use the above lemma as follows.

\begin{lemma}\label{lemma:fixed_induced_notcont}
If $\Gamma\cdot N$ is dense in $G$ and $V^N\neq 0$, then $W_N<W$ is a non-zero $\Gamma$-submodule which extends uniquely to a continuous $G$-module and descends to a continuous $G/N$-module.
\end{lemma}

\begin{proof}
In view of Lemma~\ref{lemma:fixed_induced}, it suffices to prove that $\cm V^N$ is non-zero. Observe that $V^N$ is a coefficient $G$-module and recall Remark~\ref{rem:not_ambiguous}. Since $V^N\neq 0$, the statement follows from the fact that the submodule of continuous vectors of a coefficient module is always weak-$*$ dense~\cite[3.2.3]{Monod}.
\end{proof}

\begin{remark}[{{\bfseries irreducibility}}]\label{rem:irreducibility}
Let $G$ be as for Theorem~\ref{thm:semi-simple}, let $K\lhd G$ be the product of all compact factors (cf. Remark~\ref{rem:compact_factors}) and let $\Gamma<G$ be a lattice. Margulis' criterion~\cite[II.6.7]{Margulis} shows that $\Gamma$ is irreducible if and only if $\Gamma\cdot N$ is dense in $G$ whenever $N$ is the product of $K$ with at least one isotropic factor. As pointed out in Remark~\ref{rem:compact_factors}, Theorem~\ref{thm:semi-simple}~(i) holds under the assumption that $V^N=0$ for all such $N$.
\end{remark}

\begin{proof}[Proof of Corollary~\ref{cor:restriction}]
Let $V=\lft_0(G/\Gamma)$ be the integral zero subspace of $\lft_0(G/\Gamma)$, which identifies to the quotient $\lft(G/\Gamma)/\RR$. Then $V$ is a semi-separable coefficient $G$-module, as follows either by a trivial case of Lemma~\ref{lemma:semi-separable_quotient} or by mapping it into $L^2_0(G/\Gamma)$. A combination of the long exact sequence with the induction isomorphism~\eqref{eq:induction} shows that in order to prove the corollary, it suffices to prove that $\hb^2(G;V)$ vanishes for all $n< 2\,\rank(G)$ (this is also explained in Corollary~10.1.7 of~\cite{Monod}). Since $\Gamma$ is irreducible, $V^N=0$ for every $N$ as in Remark~\ref{rem:irreducibility} (this can be seen as a trivial case of Lemma~\ref{lemma:fixed_induced_notcont}). Thus Theorem~\ref{thm:semi-simple} indeed establishes the required fact.
\end{proof}

As noted in Remark~\ref{rem:general_groups}, our results can be adapted to more general algebraic groups or Lie groups. In fact, in order to prove Corollary~\ref{cor:spaces}, some arrangements are needed in Corollary~\ref{cor:restriction} to accommodate for the fact that algebraic connectedness and simple connectedness of $\RR$-groups do not coincide with their topological counterpart for the groups of $\RR$-points. Therefore, we first propose the following variant of Corollary~\ref{cor:restriction} (which is more general than what we shall need).

\begin{corollary}\label{cor:restriction_R}
Let $L$ be any connected Lie group and let $\Gamma<L$ be an irreducible lattice. Then the restriction map
$$\hb^n(L; \RR) \lra \hb^n(\Gamma; \RR)$$
is an isomorphism for all $\ n < 2\,\rank_\RR(L/\Rad(L))$.
\end{corollary}

\noindent
(For the sake of this statement, we understand \emph{irreducible} in the \emph{ad hoc} acception that $\Gamma\cdot N$ be dense in $L$ whenever $N\lhd L$ is non-amenable.)

\begin{proof}
We first assume that $L$ is a connected semi-simple Lie group with finite centre and no compact factors (this is the case needed for Corollary~\ref{cor:spaces}). Then there is a connected semi-simple $\RR$-group $\GG$ such that $L=\GG(\RR)^+$, see~\cite[6.14]{Borel-Tits73}. Let $\pi: \ti\GG\to\GG$ be an algebraic universal cover over $\RR$, so that $\ti \GG$ is connected and simply connected (for general existence, see \emph{e.g.}~\cite[2.10]{Platonov-Rapinchuk}). Then $\GG(\RR)^+ = \pi(\ti\GG(\RR)^+)$ by~\cite[6.3]{Borel-Tits73}. On the other hand, $\ti\GG(\RR)^+ = \ti\GG(\RR)$; indeed, this is a special case of the Kneser--Tits problem but was already known to \'E.~Cartan~\cite{ECartan27}, see~\cite[4.8]{Borel-Tits72}. In any case, setting $\ti L=\ti\GG(\RR)$, we have a finite cover $\pi:\ti L \to L$ and consider the lattice $\pi^{-1}(\Gamma)$ in $\ti L$. The corresponding restriction and inflation maps form a commutative square
$$\xymatrix{
\hb^n(L;\RR) \ar[rr]^{\mathrm{rest}} \ar[d]^{\mathrm{infl}} &&\hb^n(\Gamma;\RR) \ar[d]^{\mathrm{infl}}\\
\hb^n(\ti L;\RR) \ar[rr]^{\mathrm{rest}} &&\hb^n(\ti \Gamma;\RR)
}$$
Both inflation maps are isomorphisms since the kernel of $\pi$ is finite~\cite[\S8.5]{Monod}.  By Corollary~\ref{cor:restriction}, the lower restriction map is an isomorphism when $n< 2\,\rank_\RR(\ti L)$. This finishes this case since $\Rad(L)$ is trivial and the real rank of $L$ is $\rank_\RR(\ti\GG)$.

We now consider the general case. Let $R\lhd L$ be the product of the radical $\Rad(L)$ and the compact part of the semi-simple group $L/\Rad(L)$. Then $L/R$ is as in the first case and has same real rank as $L/\Rad(L)$. Moreover, Auslander's theorem~\cite{Auslander63} implies that the group $\underline\Gamma = \Gamma/(\Gamma\cap R)$ is still a lattice in $L/R$ (in this generality it is due to H.-C.~Wang~\cite{Wang63}; see~\cite[8.27]{Raghunathan}). We have again a commutative diagram
$$\xymatrix{
\hb^n(L/R;\RR) \ar[rr]^{\mathrm{rest}} \ar[d]^{\mathrm{infl}} &&\hb^n(\underline\Gamma;\RR) \ar[d]^{\mathrm{infl}}\\
\hb^n(L;\RR) \ar[rr]^{\mathrm{rest}} &&\hb^n(\Gamma;\RR)
}$$
The inflation maps are still isomorphisms because $R$ and $\Gamma\cap R$ are amenable~\cite[\S8.5]{Monod}. This time, the upper restriction map is an isomorphism by the first case, concluding the proof.
\end{proof}

\begin{proof}[Proof of Corollary~\ref{cor:spaces}]
Let $X$ be as in the statement and $\ti X$ its universal cover. Since $X$ is ample, its fundamental group $\Gamma=\pi_1(X)$ is a lattice in the connected Lie group $L=\Is(\ti X)^\circ$. We recall the following facts (see \emph{e.g.}~\cite{Helgason01}): $L$ is semi-simple with finite centre and without compact factors; $\Gamma$ is irreducible in $L$; the real rank of $L$ is the geometric rank of $X$. Therefore, we can apply Corollary~\ref{cor:restriction_R}. The result now follows from the fact that there is an isomorphism $\hb^\bu(X)\cong \hb^\bu(\Gamma;\RR)$. The latter admits a direct proof in the present case since $\widetilde{X}$ is contractible, though it can also be seen as a trivial case of a much deeper result of M.~Gromov~\cite[p.~40]{Gromov}.
\end{proof}

\begin{proof}[Proof of Corollary~\ref{cor:lattice_semisimple}]
Suppose $\hb^n(\Gamma; W)\neq 0$ for some $n <  2\,\rank_k(\GG)$. In view of the induction isomorphism~\eqref{eq:induction}, we have $\hb^n(G;V)\neq 0$ and thus Theorem~\ref{thm:semi-simple} with Remark~\ref{rem:irreducibility} implies that there is $N\lhd G$ with $V^N\neq 0$ and $\Gamma\cdot N$ is dense in $G$. Thus Lemma~\ref{lemma:fixed_induced_notcont} finishes the proof.
\end{proof}

\section{Product groups}\label{sec:products}
There is of course no (non-trivial) Tits building structure for general groups. For products of arbitrary groups $G_i$, however, we will use a very basic variant of the idea of Tits building. The idea is to consider probability $G_i$-spaces $B_i$ as zero-dimensional simplicial complexes endowed with a measure class structure; then the join of these complexes is a high-dimensional simplicial complex, also endowed with a measure class.
\subsection{Joins}\label{sec:joins}
Let $\ B_1, \ldots, B_r$ be standard probability spaces. For any subset $I\se \{1, \ldots, r\}$ we define the probability space $B_I=\prod_{i\in I} B_i$, with $B_\vn= \{\vn\}$. Further, define the finite measure spaces
$$B^{(p)}\ =\ \bigsqcup_{|I|=p+1} B_I, \kern10mm -1\leq p\leq r-1.$$
For $I=\{i_0 < \ldots <i_p\}$ and $0\leq j\leq p$ we consider the canonical factor map $B_I\to B_{I\sm \{i_j\}}$; one obtains thus maps $\partial_{p,j}: B^{(p)}\to B^{(p-1)}$. Let now $V$ be the dual of a separable Banach space; the maps $\partial_{p,j}$ induce bounded linear maps
$$d_{p,j}: \lw(B^{(p-1)};V) \lra  \lw(B^{(p)};V), \kern10mm 0\leq j\leq p\leq r-1.$$
We define $d_p=\sum_{j=0}^p (-1)^j d_{p,j}$ and observe $\lw(B^{(-1)};V)\cong V$.

\begin{remark}\label{rem:duality}
Let $V^\flat$ be a predual of $V$ and recall that $\lw(B^{(p)};V)$ is the dual of $L^1(B^{(p)};V^\flat)$. If we still denote by $\partial_{p,j}$ the maps $L^1(B^{(p)};V^\flat) \to L^1(B^{(p-1)};V^\flat)$, then $d_{p,j}$ is the adjoint of $\partial_{p,j}$ and in particular each $d_p$ is adjoint.
\end{remark}

If we were dealing with abstract simplicial complexes, then it would be a standard fact that the cohomology of $B^{(\bu)}$ vanishes in degrees other than $0$ and $r-1$, because of standard formulas for joins (see e.g.~\cite{Munkres}). In our setting, one has to take both measurability and almost everywhere identifications into account; resorting to explicit calculation, we find that this creates no difficulties:

\begin{lemma}\label{lemma:join}
Let $\ B_1, \ldots, B_r$ be standard probability spaces and $V$ the dual of a separable Banach space. Then the sequence
$$0 \lra  V \lra \lw(B^{(0)};V) \lra \cdots \lra \lw(B^{(r-1)};V)$$
is exact.
\end{lemma}

\begin{proof}
We consider elements $f\in \lw(B^{(p)};V)$ as families $(f_I)$ where $f_I\in \lw(B_I;V)$ and $|I|=p+1$. Define
$$s_{p,j}: \lw(B^{(p)};V) \lra  \lw(B^{(p-1)};V),\kern10mm 0\leq j\leq p\leq r-1,$$
as follows:
$$(s_{p,j} f)_J\ =\ \sum_{i_{j-1} < i < i_j} \int_{B_i} f_{J\cup\{i\}},\kern10mm\text{for } J=\{i_0, \ldots, i_{p-1}\}.$$
It is understood in the above formula that the summation bounds $i_{j-1} < i < i_j$ reduce to $i<i_0$ when $j=0$, to $i_{p-1}<i$ when $j=p$ and impose no restriction when $j=p=0$. We define $s_p=\sum_{j=0}^p (-1)^j s_{p,j}$ and claim that
\begin{equation}\label{eq:homotop}
d_p\, s_p + s_{p+1}\, d_{p+1} = r\cdot\id \kern5mm\text{on } \lw(B^{(p)};V), \kern10mm -1\leq p\leq r-2,
\end{equation}
with the convention that $s_{-1}$ and $d_{-1}$ are zero. In order to verify the claim, one checks the following relations by direct calculation.
\begin{align*}
d_{p,k}\, s_{p,j} &=  s_{p+1,j}\, d_{p+1,k+1}\kern5mm&\text{for } &0\leq j\leq k-1\leq p,\\
d_{p,k}\, s_{p,k} &=  s_{p+1,k}\, d_{p+1,k+1} + s_{p+1,k+1}\, d_{p+1,k} +\id  \kern5mm&\text{for } &0\leq k\leq p,\\
d_{p,k}\, s_{p,j} &=  s_{p+1,j+1}\, d_{p+1,k}\kern5mm&\text{for } &0\leq k+1\leq j\leq p.
\end{align*}
In addition, one verifies
$$\sum_{k=0}^{p+1} s_{p+1, k}\, d_{p+1, k}\ =\ (r-(p+1))\cdot\id \kern5mm\text{on } \lw(B^{(p)};V).$$
Putting everything together yields indeed~\eqref{eq:homotop} (the case $p=-1$ is immediate). This claim finishes the proof of the lemma.
\end{proof}

In order to gain one more degree in Theorem~\ref{thm:splitting}, we shall need the following information on the end of the exact sequence of Lemma~\ref{lemma:join}.

\begin{lemma}\label{lemma:closed}
The map $d_{r-1}: \lw(B^{(r-2)};V) \to \lw(B^{(r-1)};V)$ has closed range.
\end{lemma}

\begin{proof}
We argue as for Theorem~\ref{thm:complexe:lw}, now using the notation of Remark~\ref{rem:duality}: By the closed range theorem, it suffices to prove that $\partial_{r-1}: L^1(B^{(r-1)};V^\flat) \to L^1(B^{(r-2)};V^\flat)$ is open. Since we are in the top degree $r-1$, this map is the direct sum of \emph{surjective} maps $\partial_{r-1, j}$. The latter are open by the open mapping theorem.
\end{proof}

\subsection{Ergodicity with coefficients}\label{sec:ergodic}
Let $G$ be a locally compact second countable group and $X$ a standard probability space with a measurable $G$-action (thus $G$ preserves the measure class but not necessarily the measure). Recall that the action is \emph{ergodic} if and only if every $G$-invariant measurable map $X\to\RR$ is essentially constant. In this definition, one can of course replace $\RR$ with several other spaces, for instance any separable Banach space, since only the Borel structure is of relevance. This suggests the following definition proposed by M.~Burger and the author:

\begin{definition}[\cite{Burger-Monod3,Monod}]\label{defi:double}
The $G$-action on $X$ is \emph{ergodic with coefficients} if every $G$-equivariant measurable map $X\to U$ to any separable coefficient $G$-module $U$ is essentially constant.

\noindent
(Equivalently, is it enough to consider only maps that are bounded.)
\end{definition}

We recall for the above definition that on separable dual Banach spaces the strong, weak and {weak-$*$} Borel structures all coincide (see \emph{e.g.}~\cite[3.3.3]{Monod}).

\begin{lemma}\label{lemma:ergodic}
Let $V$ be a semi-separable coefficient module. If the $G$-action on $X$ is ergodic with coefficients, then one has an identification $\lw(X;V)^G\cong V^G$.
\end{lemma}

\begin{proof}
Let $V\to U$ be an injective adjoint $G$-morphism into a separable coefficient module $U$. In particular, this map is {weak-$*$} continuous and thus induces an equivariant injection of $\lw(X;V)$ into $\lw(X;U)=\lft(X;U)$. This implies already that all elements of $\lw(X;V)^G$ are essentially constant; but their essential value is invariant by equivariance.
\end{proof}

\begin{proposition}\label{prop:ind}
Let $G=G'\times G''$ be a product of locally compact second countable groups, let $X$ be a standard probability $G''$-space and let $V$ be a semi-separable coefficient $G$-module. If the $G''$-action on $X$ is both ergodic with coefficients and amenable, then there is a natural isomorphism
$$\hb^\bu(G; \lw(X;V))\ \cong\ \hb^\bu(G'; V^{G''}).$$
\end{proposition}

\begin{proof}
Define the coefficient $G$-module $U_n= \lw({G'}^{n+1}; \lw(X;V))$ and consider the sequence
$$0 \lra \lw(X;V) \lra U_0 \lra U_1 \lra \cdots$$
with the usual homogeneous coboundary maps.  Since the $G$-action on ${G'}^{n+1}\times X$ is amenable and $U_n\cong \lw({G'}^{n+1}\times X;V)$ (\cite[2.3.3]{Monod}), each $U_n$ is a relatively injective $G$-module and therefore the complex
\begin{equation}\label{eq:resol}
0 \lra U_0^G \lra U_1^G \lra \cdots \lra U_n^G\lra\cdots
\end{equation}
realizes $\hb^\bu(G; \lw(X;V))$. On the other hand, we have
$$U_n^G\ =\ \lw({G'}^{n+1}; \lw(X;V)^{G''})^{G'} \ =\ \lw({G'}^{n+1}; V^{G''})^{G'}$$
by Lemma~\ref{lemma:ergodic}. Therefore, the complex~\eqref{eq:resol} can also be written as the non-augmented complex of $G'$-invariants of the familiar resolution
$$0 \lra V^{G''} \lra \lw({G'}; V^{G''}) \lra \lw({G'}^2; V^{G''}) \lra\cdots$$
used to compute $\hb^\bu(G'; V^{G''})$.
\end{proof}

The existence of (non-trivial) $G$-spaces that are ergodic with coefficients is not obvious; for instance, the multiplication action on $G$ itself is not ergodic with coefficients unless $G=1$. The following result is therefore useful.

\begin{theorem}[\cite{Burger-Monod3,Kaimanovich03}]\label{thm:double}
For every locally compact second countable group $G$ there exist standard probability spaces $B, B^-$ with amenable $G$-actions such that the diagonal $G$-action on $B\times B^-$ is ergodic with coefficients. Furthermore, one can choose $B$, $B^-$ isomorphic.\hfill\qedsymbol
\end{theorem}

We indicate below a short proof in the special case where $G$ is finitely generated (and hence discrete). This proof follows the original argument provided in~\cite{Burger-Monod3} for compactly generated locally compact groups; it has the peculiar feature to reduce a general statement for arbitrary finitely generated groups to classical properties of the group $\SL_2(\RR)$. The different and more general proof later provided by V.~Ka{\u\i}manovich in~\cite{Kaimanovich03} shows that in fact one can take for $B$ the \emph{Poisson} (or Poisson--Furstenberg, or ghoti) boundary of any ``spread out, non-degenerate'' random walk and for $B^-$ the boundary of the backward random walk. In any case, one may always arrange $B^-=B$, whence the terminology of \emph{double ergodicity} used in~\cite{Burger-Monod3}. The amenability of such actions was established in~\cite{Zimmer78b}.

\smallskip

Let thus $G$ be a finitely generated group. There is an epimorphism $F\to G$ for some finite rank non-Abelian free group $F$; let $N\lhd F$ be its kernel. One can realize $F$ as a lattice in $H=\SL_2(\RR)$. Consider the (isomorphic) $H$-spaces $H/P$ and  $H/P^-$, where $P$ is a minimal parabolic. In particular, those are amenable $H$-actions since $P$ is an amenable group, and thus also amenable $F$-actions since $F$ is closed in $H$. It follows that the $G$-actions on the factor spaces $B$, $B^-$ of $N$-ergodic components are also amenable. In view of the canonical factor map $H/P \times H/P^- \to B\times B^-$, it suffices now to show that the $F$-action on $X=H/P \times H/P^-$ is ergodic with coefficients. We claim that it suffices to show that the $H$-action has this property; that latter fact follows from the Bruhat decomposition and the Mautner phenomenon exactly as in Section~\ref{sec:semi-separable}, thus finishing the proof. As for the claim, let $f:X\to U$ be a bounded $F$-equivariant map as in Definition~\ref{defi:double}. Now ``induce'' it to an $H$-equivariant map $X\to L^2(H/F;U)$. If the latter is constant, then so is $f$, proving the claim.

\smallskip

Here is an elementary consequence that could perhaps also be established directly using an appropriate fixed point theorem:

\begin{lemma}\label{lemma:h1}
Let $G$ be locally compact second countable group and $V$ a semi-separable coefficient $G$-module. Then $\hb^1(G;V)$ vanishes and $\hb^2(G;V)$ is Hausdorff.
\end{lemma}

\begin{proof}
Let $B$ and $B^-=B$ be as in Theorem~\ref{thm:double} and realize $\hb^\bu(G;V)$ as the cohomology of the complex
$$ 0\lra \lw(B;V)^G \lra \lw(B^2;V)^G \lra \lw(B^3;V)^G \lra \cdots$$
By Lemma~\ref{lemma:ergodic} and the definition of the homogeneous coboundary maps, this complex starts with
$$ 0\lra V^G \xrightarrow{\ 0\ } V^G \xrightarrow{\ \epsilon\ } \lw(B^3;V)^G \lra \cdots$$
where $\epsilon$ is the inclusion of constant maps. This proves both claims since $\epsilon$ has closed range.
\end{proof}

Notice that the first statement of Lemma~\ref{lemma:h1} can be considered as the trivial case of our main results when the number of factors in a product (or the rank of a semi-simple group) is one.

\subsection{On Theorem~\ref{thm:splitting} and consequences}
We shall use the following fact, which follows from the definition of amenable actions: For a finite family of groups each given with an amenable action, the product action of the product of these groups on the product space is also amenable.

\begin{proof}[Proof of Theorem~\ref{thm:splitting}]
We use the notations of the theorem. For each $1\leq i\leq \ell$, we denote by $B_i$ and $B_{i+\ell}$ amenable $G_i$-spaces such that the diagonal $G_i$-action on $B_i\times B_{i+\ell}$ is ergodic with coefficients (by virtue of Theorem~\ref{thm:double}). We adopt now the notation of Section~\ref{sec:joins} with $r=2\ell$ and endow all spaces $B^{(p)}$ with their natural $G$-action. Defining $W_p=\lw(B^{(p)};V)$, Lemmata~\ref{lemma:join} and~\ref{lemma:closed} put us in the situation of Remark~\ref{rem:ex_sequence}. Therefore, we can apply Proposition~\ref{prop:ex_sequence} and thus conclude the proof of Theorem~\ref{thm:splitting} if we show that $\hb^q(G;W_p)$ vanishes for all $q\geq 0$ and all $0\leq p \leq r-1$.

Since $\hb^\bu(G;-)$ commutes with finite direct sums~\cite[8.2.10]{Monod}, it suffices to show that $\hb^q(G; \lw(B_I;V))$ vanishes whenever $|I|=p+1$. We write $G=G'\times G''$, where $G''$ is the product of all $G_i$ with either $i$ or $i+\ell$ in $I$, and $G'$ the product of the remaining factors. The $G''$-action on $B_I$ is amenable and we claim that it is ergodic with coefficients. Indeed, let $U$ be a separable coefficient $G''$-module; we need to show that every element in $\lw(B_I;U)^{G''}$ is constant. There is a natural identification of coefficient $G''$-modules~\cite[2.3.3]{Monod}
$$\lw(B_I;U)\ \cong\ \lw(X_{i_p};\lw(X_{i_{p-1}}; \ldots; \lw(X_{i_0};U)\ldots )),$$
where $I=\{i_0, \ldots, i_p\}$. Using successively the invariance under $G_{i_0}$ (or $G_{i_0 - \ell}$), then $G_{i_1}$ (or $G_{i_1 - \ell}$), and so on, we find that $G''$-invariant elements are constant, establishing the claim. We are thus in the setting of Proposition~\ref{prop:ind} and deduce
\begin{equation}\label{eq:thm:splitting}
\hb^q(G; \lw(B_I;V))\ \cong\ \hb^q(G'; V^{G''}).
\end{equation}
Since $I$ contains at least one element, there is at least one factor occurring in $G''$; therefore $V^{G''}=0$ and the above vanishes as was to be shown.
\end{proof}

\begin{remark}\label{rem:thm:splitting}
Recall~\eqref{eq:hb_cm} that the right hand side in~\eqref{eq:thm:splitting} is identical with $\hb^q(G'; \cm_{G'}V^{G''})$, and of course
$$\cm_{G'}V^{G''} = \cm_G V^{G''} \se \cm_G V^{G_i}$$
for any factor $G_i$ of $G''$. Therefore, we established the conclusion of Theorem~\ref{thm:splitting} under the slightly weaker assumption that $V$ is a semi-separable coefficient $G$-module with $\cm V^{G_i}=0$ for all $\ 1\leq i\leq \ell$.
\end{remark}

\begin{proof}[Proof of Corollary~\ref{cor:split_lattice}]
We use the notations of Section~\ref{sec:lattices} and recall that the induced $G$-module $V$ is semi-separable, see Lemma~\ref{lemma:ind:lattices}. Suppose that $W_{G_j}$ vanishes for all $j$. By Lemma~\ref{lemma:fixed_induced}, this implies $\cm V^{G_i}=0$ for all $\ 1\leq i\leq \ell$. In view of the refinement of Theorem~\ref{thm:splitting} established in the above Remark~\ref{rem:thm:splitting}, we deduce that $\hb^n(G; V)$ vanishes for all $n<2\ell$. (Instead of using Remark~\ref{rem:thm:splitting}, one can replace Lemma~\ref{lemma:fixed_induced} by Lemma~\ref{lemma:fixed_induced_notcont}.) By the induction isomorphism~\eqref{eq:induction}, it follows that $\hb^n(\Gamma; W)$ vanishes in the same range, finishing the proof.
\end{proof}

\begin{proof}[Proof of Corollary~\ref{cor:rstriction_product}]
The proof of Corollary~\ref{cor:restriction} applies word for word, replacing only Theorem~\ref{thm:semi-simple} with Theorem~\ref{thm:splitting}.
\end{proof}

\subsection{Global fields and ad\`eles}
Let $K$ be a global field and $\GG$ a connected simply connected almost $K$-simple $K$-group. Recall that the ring $\AAA_K$ of ad\`eles is the restricted product of the completions $K_v$ where $v$ ranges over the set $\SV$ of places of $K$. We write $\Gamma=\GG(K)$ and $G=\GG(\AAA_K)$ and recall that $\Gamma$ is a lattice in $G$ by results of A.~Borel and Behr--Harder (see~\cite[I.3.2.2]{Margulis}).

For any $\SU\se\SV$ we denote by $G_\SU$ the direct factor of $G$ obtained as restricted product over $\SU$. Let $\SA\se\SV$ be the set of places $v$ such that $\GG$ is $K_v$-anisotropic and set $\SI=\SV\sm\SA$, so that $G=G_\SA\times G_\SI$. Recall that $\SA$ is finite (see \emph{e.g.}~\cite[4.9]{Springer77}) and that $G_\SA$ is compact. The strong approximation theorem (see \emph{e.g.}~\cite[II.6.8]{Margulis}) states that $\Gamma\cdot G_\SU$ is dense in $G$ as soon as $\SU\se \SV$ contains at least one element of $\SI$.

\begin{proof}[Proof of Theorem~\ref{thm:adelic_rigidity}]
Let $W$ be a semi-separable coefficient $\GG(K)$-module and assume that $\hb^n(\GG(K);W)\neq 0$ for some $n$. Let $\ell>n/2$ and choose a partition $\SV=\SV_1\sqcup\ldots\sqcup\SV_\ell$ such that each $\SV_i$ contains at least an element of $\SI$. In view of the preceding discussion, we are contemplating an irreducible lattice
\begin{equation}\label{eq:decompose_adelic}
\Gamma<G = G_{\SV_1} \times \cdots \times G_{\SV_\ell}.
\end{equation}
Therefore, Corollary~\ref{cor:split_lattice} applies and proves the theorem.
\end{proof}

Suppose that we choose the above decomposition of $\SV$ in such a way that each $\SV_i$ contains infinitely many places. Then, when Corollary~\ref{cor:split_lattice} states that for some $i$ the module $W_{G_{\SV_i}}$ is non-zero, it provides us indeed with a $G$-representation which is trivial on infinitely many local factors $\GG(K_v)$.

\begin{proof}[Proof of Corollary~\ref{cor:adelic_restriction}]
We perform the same decomposition~\eqref{eq:decompose_adelic} and then apply Corollary~\ref{cor:rstriction_product}.
\end{proof}

\section{Further considerations}
This section presents additional material not needed for the body of the article and therefore the presentation is more elliptical.

\subsection{Towards stabilisation}
We would like to suggest how certain results of this paper could perhaps be used in the study of \emph{stabilisation}. Given a sequence $G_n \se G_{n+1}\se \ldots$ of classical groups $G_n$ with natural inclusions, the question is whether and in what range the corresponding restriction maps are isomorphisms for cohomology with trivial coefficients. The classical situation is well understood and is of importance notably \emph{via} the cohomology of the limiting object $G_\infty$, which plays a r\^ole in topology (see \emph{e.g.}~\cite{Cartan61} and~\cite[\S12]{Borel74}). In bounded cohomology, some rather limited information is known for $\SL_n$, see~\cite{MonodJAMS}.

\medskip

Let $G$ be a general semi-simple group as in Section~\ref{sec:tits} and keep the notation introduced there; in particular, $r=|S|=\rank(G)$. For any $I\se S$ we consider the (generalised) L\'evi decomposition $P_I= V_I \rtimes \SZ_G(T_I)$ as introduced in Section~\ref{sec:vanish_ss}. We denote by $G_I$ the quotient of $\SZ_G(T_I)$ by its centre and call it the \emph{semi-simple part} of $P_I$.

\begin{theorem}\label{thm:spectral}
There is a first quadrant spectral sequence $\ee_\bu^{\bu, \bu}$ converging to zero below degree $r$ such that
$$\ee_1^{p, q}\ \cong\ \bigoplus_{|I| = r-p}\hb^q(G_I) \kern1cm(\forall\,0\leq p\leq r,\,\forall\, q).$$
Moreover, these isomorphisms intertwine the differential with signed sums of restriction maps.
\end{theorem}

The convergence ``below degree $r$'' means $\ee^{p,q}_\infty = 0$ for all $p+q<r$. As for the last statement, it is to be precised as follows: there are indeed inclusions $\SZ_G(T_I)\se \SZ_G(T_J)$ whenever $I\se J$; the corresponding restriction map descends to a map $\hb^\bu(G_J)\to \hb^\bu(G_J)$ via the inflation isomorphisms.

\medskip

In order to use Theorem~\ref{thm:spectral} for stabilisation, one probably ought to use the fact that these restrictions are modelled on the acyclic simplex of proper subsets of $S$ with dual inclusion. (For instance, this fact shows \emph{a contrario} that the differential of Theorem~\ref{thm:spectral} must vanish on $\ee_1$ at those indices where stabilisation holds.)

\begin{proof}[Proof of Theorem~\ref{thm:spectral}, sketch]
Consider the double complex defined by
$$\lft(G^{p+1} \times \ST^{(q-1)})^G \kern1cm(p,q\geq 0)$$
with $\ST^{(-1)}=G/G$ and the usual convention $\lft(\vn)=0$ for $q>r$. Arguing exactly as in~\cite{MonodJAMS},\cite{MonodMRL}, one can apply Theorem~\ref{thm:complexe:lw} and~\cite[8.2.5]{Monod} to establish that one of the two associated spectral sequences converges to zero below degree $r$. (More precisely, everywhere except in bi-degree $(p,q)$ with $q=r$.) One of the few differences is that we chose here to include the empty simplex $\ST^{(-1)}$ so that we are really considering the spectral sequences resulting from applying the functor $\hb^\bu(G,-)$ to the \emph{augmented} complex of Theorem~\ref{thm:complexe:lw}. The other spectral sequence, whilst it abuts to the same result up to grading (which is irrelevant for vanishing), has the following first tableau:
$$\ee^{p,q}_1 = \hb^q\big(G;\lft(\ST^{(q-1)})\big) \cong  \bigoplus_{|I| = r-p}\hb^q(P_I) \cong \bigoplus_{|I| = r-p}\hb^q(G_I).$$
The statement about the differentials is established by a calculation using~10.1.7 and~10.1.2(v) in~\cite{Monod} and the fact that we have commuting restrictions and inflations.
\end{proof}

The bounds in Theorem~\ref{thm:spectral} are certainly not optimal since we could use the group $\ti G=G\times G$ as in the proof of the vanishing results. However, the interpretation of the differentials needs then to be adapted.

\subsection{Application to ergodic theory}
The orbit-equivalence rigidity results of~\cite{Monod-Shalom2} make extensive use of the \emph{second} bounded cohomology. More precisely, the class $\mathcal{C}$ of all countable groups $\Gamma$ with non-vanishing $\hb^2(\Gamma;V)$ for some mixing unitary representation $V$ is considered. On the one hand, this class contains all groups that are \emph{negatively curved} in a rather general sense. On the other hand, the methods introduced in~\cite{Burger-Monod3},\cite{Monod} provide an array of tools to analyse $\hb^2$, including a ``splitting'' result implying the degree two case of our Theorem~\ref{thm:splitting}.

\smallskip

A careful review of the proofs given in~\cite{Monod-Shalom2} shows that they can be extended to a more general case using Theorem~\ref{thm:splitting}. Let thus $\mathcal{C}^n$ be the class of countable groups $\Gamma$ with non-vanishing $\hb^p(\Gamma;V)$ for some $p\leq n$ and some mixing unitary representation $V$. As a first example, here is an extension of the ``prime factorization'' phenomenon (see the discussion of Theorem~1.16 in~\cite{Monod-Shalom2}). We refer to that paper for terminology.

\begin{theorem}\label{thm:me}
Let $\Gamma = \Gamma_1 \times \cdots \times \Gamma_n$ and $\Lambda = \Lambda_1 \times \cdots \times \Lambda_n$ be products of torsion-free (infinite) countable groups. Assume that all the $\Gamma_i$ are in~$\mathcal{C}^{2n-1}$.

\nobreak
If $\Gamma$ is measure-equivalent to $\Lambda$, then after permutation of the indices $\Gamma_i$ is measure-equivalent to $\Lambda_i$ for all $i$.
\end{theorem}

The discussion of this section has to be taken with a \emph{caveat} since we do not know examples showing that the inclusions
$$\mathcal{C}=\mathcal{C}^2\se \cdots \se \mathcal{C}^n\se \mathcal{C}^{n+1}\se \cdots$$
are proper. In the opposite direction, however, the results of the introduction show that $\Gamma\notin \mathcal{C}^n$ when $\Gamma$ is either a lattice in a semisimple group of rank~$>n/2$  or a product of more than $n/2$ infinite factors (or a lattice in such a product). Further, a group $\GG(K)$ as in Section~\ref{sec:intro_adelic} does not belong to any $ \mathcal{C}^n$ at all.

\medskip

Our second example extends Corollary~2.21 in~\cite{Monod-Shalom2}.

\begin{theorem}\label{thm:oe}
Let $\Gamma=\Gamma_1\times \cdots\times \Gamma_n$ be a product of torsion-free groups in~$\mathcal{C}^{2n-1}$ and let $Y$ be a mildly mixing $\Gamma$-space.

\nobreak
If this action is orbit-equivalent to any mildly mixing action of any torsion-free countable group $\Lambda$ on a probability space $X$, then there is an isomorphism $\Lambda\cong\Gamma$. Moreover, the actions on $X$ and $Y$ are then isomorphic.
\end{theorem}

\begin{proof}[On the proofs of Theorems~\ref{thm:me} and~\ref{thm:oe}]
We only indicate what needs to be added to the arguments in~\cite{Monod-Shalom2}. Let $V$ be a mixing unitary $\Gamma_i$-representation and view it also as a $\Gamma$-module. We first observe that the inflation map
$$\hb^\bu(\Gamma_i;V)\lra \hb^\bu(\Gamma;V)$$
in injective in all degrees since the restriction provides it with a left inverse. Given a measure- (or orbit-) equivalence, one obtains by \emph{induction} a semi-separable coefficient $\Lambda$-module $W$ of the form $W=\lw(\Omega;V)$ described in Section~\ref{sec:semi-separable}. Whilst $L^2$-induction required us to work in degree two in~\cite{Monod-Shalom2}, it was shown there that there is an injective map in any degree for $\lft$-induction
$$\hb^\bu(\Gamma;V) \lra \hb^\bu(\Lambda;W),$$
see the proof that Proposition~4.5 implies Theorem~4 in~\cite{Monod-Shalom2}. Our Theorem~\ref{thm:splitting} applies precisely in the generality of semi-separable modules such as this $\lft$-induced module and for the right degrees. Now the proofs continue unchanged; one uses that elements of $W$ are in fact strongly measurable as maps $\Omega\to V$. A last remark on the assumptions of Theorem~\ref{thm:oe} is that mildly mixing actions of torsion-free groups are automatically essentially free.
\end{proof}


\end{document}